\documentclass{amsart}
\usepackage[utf8]{inputenc}

\usepackage{amsthm}
\usepackage{amsmath}
\usepackage{amssymb}
\usepackage[all]{xy}
\usepackage{enumerate}
\usepackage{tikz}
\usepackage{longtable}
\usepackage{xcolor}
\usepackage{dsfont}
\usepackage{tensor}
\usetikzlibrary{decorations.markings}

\usepackage{color}

\newcommand\Z{{\mathbb Z}}
\newcommand\Q{{\mathbb Q}}

\newcommand\C{{\mathbb C}}
\renewcommand\P{{\mathbb P}}
\newcommand\D{{\mathbb D}}
\newcommand\bG{{\mathbb G}}
\renewcommand\H{{\mathbb H}}

\newcommand\V{{\mathbb V}}

\newcommand\E{{\mathcal E}}

\newcommand\M{{\mathcal M}}
\newcommand\cO{{\mathcal O}}
\newcommand\U{{\mathcal U}}

\newcommand\cH{{\mathcal H}}

\newcommand\cL{{\mathcal L}}
\newcommand\cP{{\mathcal P}}
\newcommand\cV{{\mathcal V}}
\newcommand\cU{{\mathcal U}}

\DeclareMathAlphabet\mathbfcal{OMS}{cmsy}{b}{n}

\newcommand\g{{\mathfrak g}}
\newcommand\h{{\mathfrak h}}

\newcommand\p{{\mathfrak p}}

\newcommand\e{{\epsilon}}
\newcommand\w{{\omega}}

\newcommand\G{{\Gamma}}

\newcommand\zetabar{{\overline{\zeta}}}

\newcommand\vv{{\vec{\mathsf v}}}
\newcommand\ww{{\vec{\mathsf w}}}

\newcommand\bP{{\boldsymbol{\cP}}}

\renewcommand\aa{{\mathbf a}}
\newcommand\bb{{\mathbf b}}
\newcommand\ee{{\mathbf e}}

\newcommand\bt{{\mathbf t}}

\newcommand\bx{\mathbf{x}}

\newcommand\bX{\mathbf{X}}
\newcommand\bY{\mathbf{Y}}
\newcommand\bL{\mathbf{L}}
\renewcommand\v{\mathbf{v}}

\newcommand\bmu{{\boldsymbol{\mu}}}

\renewcommand\sl{\mathfrak{sl}}

\newcommand\SL{{\mathrm{SL}}}

\newcommand\Gm{{\mathbb{G}_m}}

\newcommand\DR{{\mathrm{DR}}}

\newcommand\Zag{\mathrm{Zag}}
\newcommand\PD{\mathrm{PD}}
\newcommand\LR{\mathrm{LR}}

\newcommand\sgn{\mathrm{sgn}}
\newcommand\lcm{\mathrm{lcm}}
\newcommand\can{\mathrm{can}}

\newcommand\bs{\backslash}
\newcommand\bbs{{\bs\negthickspace \bs}}

\renewcommand\ll{\langle\langle}
\newcommand\rr{\rangle\rangle}

\newcommand\KZB{\mathrm{KZB}}
\newcommand\KZ{\mathrm{KZ}}

\newcommand{\piun}{\pi_1^{\mathrm{un}}}

\newcommand\ad{\operatorname{ad}}
\newcommand\Ad{\operatorname{Ad}}

\newcommand\Spec{\operatorname{Spec}}
\newcommand\Hom{\operatorname{Hom}}

\newcommand\End{\operatorname{End}}
\newcommand\Aut{\operatorname{Aut}}
\newcommand\Der{\operatorname{Der}}

\newcommand\Gr{\operatorname{Gr}}

\newcommand\Res{\operatorname{Res}}

\newcommand\Lie{\operatorname{Lie}}

\newcommand\topo{\mathrm{top}}

\renewcommand\Im{\operatorname{Im}}

\newcommand{\mat}[1]{\begin{pmatrix} #1 \end{pmatrix}}

\numberwithin{equation}{section}

\newtheorem{theorem}{Theorem}[section]
\newtheorem{lemma}[theorem]{Lemma}
\newtheorem{prop}[theorem]{Proposition}
\newtheorem{cor}[theorem]{Corollary}
\newtheorem{bigtheorem}{Theorem}
\newtheorem{bigprop}[bigtheorem]{Proposition}

\theoremstyle{definition}
\newtheorem{definition}[theorem]{Definition}
\newtheorem{example}[theorem]{Example}

\theoremstyle{remark}
\newtheorem{remark}[theorem]{Remark}

\makeatletter
\@namedef{subjclassname@2020}{\textup{2020} Mathematics Subject Classification}
\makeatother

\title[KZB in higher level]{The universal elliptic KZB connection in higher level}
\author{Eric Hopper}
\email{eric.hopper@rochester.edu}
\address{Department of Mathematics, University of Rochester, Rochester, NY, USA}
\date{\today}

\begin{document}

\maketitle

\begin{abstract}
    The level $N$ elliptic KZB connection is a flat connection over the universal elliptic curve in level $N$ with its $N$-torsion sections removed. Its fiber over the point $(E,x)$ is the unipotent completion of $\pi_1(E - E[N],x)$. It was constructed by Calaque and Gonzalez. In this paper, we show that the connection underlies an admissible variation of mixed Hodge structure and that it degenerates to the cyclotomic KZ connection over the singular fibers of the compactified universal elliptic curve. These are the first steps in a larger project to compute the action of the Galois group of mixed Tate motives unramified over $\Z[\bmu_N,1/N]$ on the unipotent fundamental group of $\P^1 - \{0,\bmu_N,\infty\}$ and to better understand Goncharov's higher cyclotomy. 
\end{abstract}

\section{Introduction}

In this paper, we give a concise exposition of the universal elliptic KZB connection\footnote{Named for physicists Knizhnik, Zamolodchikov, and Bernard \cite{KZe, bernard}.} in level $N \ge 1$, clarify some of its properties, and prove that it underlies an admissible variation of mixed Hodge structure. 

Suppose that $N$ is a positive integer. Denote the modular curve parametrizing elliptic curves with full level $N$ structure by $Y(N)$\label{not:YN} and the universal elliptic curve over it by $\E_N$.\label{not:EN} These will be regarded as complex analytic varieties (or orbifolds when $N < 3$). Set $\E_N' := \E_N - \E_N[N]$.\label{not:ENN} The fiber $E'$\label{not:Eprime} of the projection $\E_N' \to Y(N)$ over the moduli point of the elliptic curve $E$ is the elliptic curve $E$ itself with its $N$-torsion subgroup $E[N]$ removed.

To the point $(E',x)$ in $\E_N'$ we can associate the Lie algebra $\p(E',x)$ of the unipotent completion of $\pi_1(E',x)$. It is a free pronilpotent Lie algebra over $\Q$ on $N^2+1$ generators.

Set
\begin{equation}
    \label{eqn:pN}
    \p_N := \bL\left(\bX,\bY, \bt_\alpha \mid \alpha \in \E_N[N]\right)^\wedge \Big/ \Big(\sum_\alpha \bt_\alpha = [\bX,\bY]\Big).
\end{equation}
\label{not:XY}This is a free pronilpotent Lie algebra\label{not:pN} over $\C$ of rank $N^2+1$. The level $N$ universal elliptic KZB connection is a flat connection on a (pro-)holomorphic vector bundle over $\E_N'$ whose fibers are isomorphic to $\p_N$. Its restriction to the fiber $E_\tau'$ of $\E_N' \to Y(N)$, where $E_\tau = \C/(\Z \oplus \tau \Z)$, induces an isomorphism
\begin{equation}
    \label{eqn:fibercomp}
    \psi_\tau : \p(E_\tau',x)\otimes_\Q \C \to \p_N.
\end{equation}
When $N=1$, an explicit formula for this connection was derived independently by Calaque, Enriquez, and Etinghof \cite{CEE} and by Levin and Racinet \cite{LR}. Calaque and Gonzalez \cite{CG} then gave a formula for the connection for general $N \ge 1$.

The pullback of the level $N$ KZB connection to the universal covering $\h\times \C$ of $\E_N$ is a meromorphic connection on the trivial bundle
$$
\h\times \C \times \p_N \to \h \times \C
$$
of the form
$$
\nabla = d + \Omega_N,
$$
where $\h$\label{not:h} denotes the upper half plane and
$$
\Omega_N \in \Omega^1(\h \times \C, \log \Lambda_N) \hat{\otimes} \Der \p_N.
$$
Here $\Lambda_N$ denotes the preimage of $\E_N[N]$ in $\h\times \C$. The coefficients of the 1-form $\Omega_N$ are Eisenstein series and Jacobi forms of level $N$. The connection is holomorphic on $\E_N'$ and has logarithmic singularities along the $N$-torsion sections $\E_N[N]$ and along the singular fibers of the extension of $\E_N$ to its natural smooth compactification.

The topological significance of the KZB connection is described in the following theorem. Denote by $\bP^\topo_N$ the $\Q$-local system of pronilpotent Lie algebras over $\E_N'$ whose fiber over the moduli point of $(E,x)$ is the Lie algebra $\p(E',x)$.

\begin{bigtheorem}
\label{bigthm:bundles}
The KZB connection over $\E_N'$ is naturally isomorphic to the flat vector bundle $\bP^\topo_N\otimes_\Q \cO_{\E_N'}$.
\end{bigtheorem}

So one can think of the KZB connection as the de~Rham analogue of the Betti local system $\bP^\topo_N$ and this isomorphism as a Betti to de~Rham comparison isomorphism. This point of view is implicit in previous work especially Hain's exposition of the $N = 1$ case \cite{hain:kzb}.

The Lie algebra $\p_N$ has natural Hodge and weight filtrations defined by giving $\bX$ and $\bY$ weight $-1$, each $\bt_\alpha$ weight $-2$, and placing $\bX$ in $F^0\p_N$ and $\bY$ and each $\bt_\alpha$ in $F^{-1}\p_N$. The isomorphism (\ref{eqn:fibercomp}) gives $\p_N$ a $\Q$-structure, and thus a mixed Hodge structure (MHS). This induces a MHS on $\p(E',x)$ by pulling back the Hodge and weight filtrations along \eqref{eqn:fibercomp}. We establish that this MHS is in fact the canonical one constructed in \cite{hain:dht}.

\begin{bigprop}
\label{bigprop:fiber}
For every elliptic curve $E_\tau$, the MHS on $\p(E_\tau',x)$ induced by the isomorphism $\psi_\tau$ in (\ref{eqn:fibercomp}) is the same as the canonical MHS on $\p(E_\tau',x)$. 
\end{bigprop}

More generally, the KZB connection gives an explicit construction of the corresponding variation of MHS over $\E_N'$.
\begin{bigtheorem}
\label{bigthm:avmhs}
With the $\Q$-structure from Theorem 1 and the Hodge and weight filtrations above, the local system $\mathbfcal{P}_N^\topo \to \E_N'$ is a pro-object of the category of admissible variations of mixed Hodge structure over $\E_N'$.
\end{bigtheorem}

These results are elliptic analogues of well-known facts in genus 0. The level $N$ elliptic KZB connection generalizes the cyclotomic KZ connection on the trivial $\bL(\ee_0,\ee_\zeta \mid \zeta \in \bmu_N)^\wedge$-bundle over $\Gm - \bmu_N$.\label{not:Gm}\label{not:bmu} The cyclotomic KZ connection, introduced in \cite{enr:kz}, is defined by
\begin{equation}
    \label{eqn:KZ}
    \nabla_\KZ = d - \ee_0 \frac{dw}{w} - \sum_{\zeta \in \bmu_N}\ee_\zeta\frac{dw}{w-\zeta},
\end{equation}
where each $\ee_r$ acts on $\bL(\ee_0,\ee_\zeta \mid \zeta\in \bmu_N)^\wedge$ by the adjoint action. The monodromy action induces an isomorphism of MHS
\begin{equation}
    \label{eqn:kzcomp}
    \p(\Gm - \bmu_N,x) \otimes_\Q \C \to \bL(\ee_0,\ee_\zeta \mid \zeta \in \bmu_N)^\wedge,
\end{equation}
where each $\ee_r$ is of type $(-1,-1)$. This is the genus 0 analogue of \eqref{eqn:fibercomp}.

By viewing $\Gm - \bmu_N$ as a degenerate elliptic curve with its $N$-torsion removed, we make the following observation. 

\begin{bigtheorem}
\label{bigthm:KZ}
Along each component of the singular fibers of $\overline{\E}_N$, the level $N$ KZB connection degenerates to the cyclotomic KZ connection \eqref{eqn:KZ} and, on the identity component, the isomorphism $\psi_\tau$ in \eqref{eqn:fibercomp} degenerates to \eqref{eqn:kzcomp}.
\end{bigtheorem}

This degeneration is well-known, especially in the $N = 1$ case. Taking a particular quotient yields the elliptic polylogarithm variation of Beilinson and Levin \cite{BL}, which at the cusp of $\overline{\M_{1,1}}$ degenerates to the classical polylogarithm variation. Notable areas of application include special values of the Riemann zeta function \cite{HK} and mixed Tate motives \cite[\S4]{hain:mem}. 

The degeneration in Theorem \ref{bigthm:KZ} gives an explicit formula for the change of variables $\bL(\ee_0,\ee_\zeta\mid \zeta \in \bmu_N)^\wedge \to \p_N$ induced by the inclusion of $\Gm - \bmu_N$ into a singular fiber of $\overline{\E}_N$. This is a level $N$ generalization of Hain's map in \cite[\S18]{hain:kzb}. The $\G_1(N)$ case has an especially nice formula. 

\begin{bigtheorem}
Let $q = e^{2\pi i \tau}$. The inclusion $\Gm -\bmu_N$ into $E_{\partial/\partial q}$, the first-order smoothing of the singular fiber of $\overline{\E}_{\G_1(N)}$ in the direction of $\partial/\partial q$, induces the map on completed Lie algebras $\bL(\ee_0,\ee_\zeta\mid \zeta \in \bmu_N)^\wedge \to \p_N$ given by 
\begin{equation*}
    \left\{ 
    \begin{array}{lll}
        \ee_0 & \longmapsto & \frac{\bX}{e^{\bX} - 1} \cdot \bY \\
        \ee_\zeta & \longmapsto & \bt_\zeta.
    \end{array}
    \right.
\end{equation*}
\end{bigtheorem}

A somewhat more complicated formula is required in the full level $N$ case. It is given in \S\ref{sec:res}.

The construction of the elliptic KZB connection presents two challenges not present in genus 0. The first is that, unlike in genus 0, elliptic curves have nontrivial moduli. The second is that Deligne's canonical extension of $\bP_N^\topo$ to $\overline{\E}_N$ is nontrivial as a holomorphic vector bundle, unlike in the KZ connection.

Our interest in the level $N$ KZB connection is motivated by potential applications to understanding the action of the motivic Galois group on cyclotomic multizeta values and, more generally, to Goncharov's higher cyclotomy \cite{gonch:sym}. In particular, we are developing KZB as a tool for understanding the motivic Galois action on the unipotent fundamental group of $\P^1 - \{0,\bmu_N,\infty\}$. This is established in depth 1 in the author's Ph.D. thesis \cite{thesis}. It will also be the subject of a forthcoming paper. 

By GAGA, Deligne's canonical extension of the elliptic KZB local system to $\overline{\E}_N$ is algebraic. Moreover, the level $N$ KZB connection is defined over $\Q(\bmu_N)$. This fact follows from a more general result of Chiarellotto, Proietto, and Shiho \cite{CPS}, which builds on the work of Lazda \cite{lazda}. Additionally, Hain \cite[\S4]{hain:kzb} and Luo \cite{rome} give concrete algebraic formulas in the $N = 1$ case. These facts regarding the algebraic nature of KZB are important in the study of (mixed) elliptic motives \cite{hain:mem}.

In this paper, sections 3 through 5 review modular curves, Eisenstein series, and unipotent completion. Sections 6 and 7 contain an explicit description of the KZB connection in level $N$ adapted from \cite{CG}. Sections 8 and 9 include proofs of \eqref{eqn:fibercomp} and Theorem \ref{bigthm:bundles}, respectively. In section 10, we compute the restriction of the KZB connection to first order neighborhoods of the boundary divisor of $\E_N'$ in $\overline{\E}_N$. Section 11 includes proofs of Proposition \ref{bigprop:fiber} and Theorem \ref{bigthm:avmhs}, our main result. Section 12 describes the KZB connection over other modular curves such as $Y_1(N)$. In the appendix, we prove the KZB connection is invariant under the action of $\SL_2(\Z) \ltimes \Z^2$. This argument is superficially different than those of \cite{CEE,gonza,CG} and inspired by \cite{LR}. \\\\

\noindent {\em Acknowledgements:} These results appeared in my Ph.D. thesis completed under the supervision of Richard Hain at Duke University. I am grateful for his introducing the KZB connection to me, guidance in approaching these problems, and many suggestions in the writing of this manuscript. 

\section{Notation and conventions}

We work in the category of complex analytic varieties unless otherwise noted.

Throughout, we denote by $\gamma$ the matrix $\begin{pmatrix} a & b \cr c & d \end{pmatrix} \in \SL_2(\Z)$, and $a$, $b$, $c$, and $d$ will refer to its entries. 

We will use the topologist's convention for composition of paths. If $X$ is a topological space, $\alpha,\beta : [0,1] \to X$, and $\alpha(1) = \beta(0)$, then $\alpha\beta$ denotes the path by first proceeding along $\alpha$ and then along $\beta$.

The one-dimensional pure $\Q$-Hodge structure of type $(-n,-n)$ will be denoted by $\Q(n)$. Its $\Q$-Betti and $\Q$-de~Rham generators are $\Q\ee^B$ and $\Q\ee^\DR$, respectively. The Betti to de~Rham comparison isomorphism takes $\ee^B$ to $(2\pi i)^n \ee^\DR$.

Suppose that $F$ is a field of characteristic 0 and that $V$ is a finite dimensional vector space over $F$. Denote by $\bL(V)$ the free Lie algebra on $V$. Recall that the universal enveloping algebra of $\bL(V)$ is the tensor algebra $T(V)$. Let $\bL(V)^\wedge$ be the completion of $\bL(V)$ with respect to its lower central series. Let $T(V)^\wedge$ be the completion of $T(V)$ with respect to powers of the augmentation ideal $I = \ker \psi$ where $\psi : T(V) \to F$ and $\psi(v) = 0$ for all $v \in V$.

Denote the free associative $F$-algebra generated by the finite set $\{\ee_j \mid j \in J\}$ by 
$$
F\langle\ee_j \mid j \in J\rangle.
$$
Let $I$ be the ideal generated by $\{\ee_j \mid j \in J\}$. The $I$-adic completion of $F\langle \ee_j \mid j \in J\rangle$ is the non-commutative power series ring
$$
F\ll \ee_j \mid j \in J\rr. 
$$
If $\{\ee_j \mid j \in J\}$ is a basis of $V$, then there are canonical isomorphisms
$$
F\langle \ee_j : j\in J\rangle \cong T(V) \quad \text{and} \quad F\ll \ee_j : j\in J\rr \cong T(V)^\wedge.
$$
We have similar notation for Lie algebras. The free Lie algebra over $F$ generated by $\{\ee_j \mid j \in J\}$ is denoted by 
$$
\bL(\ee_j \mid j \in J).
$$
The completion with respect to the lower central series is denoted by
$$
\bL(\ee_j \mid j \in J)^\wedge. 
$$
If $\{\ee_j \mid j \in J\}$ is a basis of $V$, then there are canonical isomorphisms 
$$
\bL(\ee_j \mid j \in J) \cong \bL(V) \quad \text{and} \quad \bL(\ee_j \mid j \in J)^\wedge \cong \bL(V)^\wedge. 
$$

The adjoint action of an element $x \in T(V)$ on $y \in \bL(V)$ will be denoted $x \cdot y$. This notation also applies to completions. If $a_n \in F$ and $x\in V$, then 
$$
\left(\sum_{n = 0}^\infty a_nx^n\right)\cdot y := \sum_{n = 0}^\infty a_n \ad_x^n(y).
$$
Also, when it is clear we are working in the derivation algebra $\Der \g$ of a Lie algebra $\g$ with trivial center, such as a free Lie algebra of rank $>1$, we will view $\g$ as a subalgebra of $\Der \g$ via the adjoint action $\ad : \g \to \Der \g$.

\section{Modular curves and universal elliptic curves}

\subsection{Framed elliptic curves}
\label{sec:framed}

A {\em framed elliptic curve} is an elliptic curve $(E,P)$ together with a basis $\{\aa,\bb\}$ of $H_1(E,\Z)$ such that the intersection number $\aa \cdot \bb$ is 1. The moduli space of framed elliptic curves is the upper half plane $\h$. The point $\tau \in \h$ corresponds to the elliptic curve $(E_\tau,0) := (\C/\Lambda_\tau,0)$ where $\Lambda_\tau := \Z \oplus \tau \Z$ with framing $\aa$ and $\bb$ determined by the images of 1 and $\tau$, respectively, under the natural isomorphism $\Lambda_\tau \to H_1(E_\tau,\Z)$. Conversely, if $(X,P)$ is an elliptic curve and $\{\aa,\bb\}$ is a symplectic framing of $H_1(X,\Z)$, then $\tau = \int_\bb \omega/\int_\aa \w$, where $\w$ is any nonzero abelian differential on $X$. The map $f: (X,P) \to (E_\tau,0)$ given by $f : x \mapsto \int_P^x \omega/\int_\aa \omega$ is an isomorphism.

The group $\SL_2(\Z)$ acts on framings on the left via
\begin{equation}
    \label{eqn:framing}
    \gamma : \begin{pmatrix} \bb \cr \aa \end{pmatrix} \longmapsto \begin{pmatrix} a & b \cr c & d \end{pmatrix} \begin{pmatrix} \bb \cr \aa \end{pmatrix}.
\end{equation}
The corresponding action on the upper half plane is
\begin{equation}
    \label{eqn:modular}
    \gamma : \tau \longmapsto \frac{a\tau + b}{c\tau + d}.
\end{equation}
The orbifold quotient $\SL_2(\Z) \bbs \h$ is the moduli space of elliptic curves $\M_{1,1}$.

\subsection{The local system $\H$}
\label{sec:H}

Denote by $\H$\label{not:H} the local system over $\M_{1,1}$ whose fiber over the moduli point of an elliptic curve $E$ is $H_1(E,\Q)$. We can pull back $\H$ to $\h$ and $\D^\ast$ along
$$
\xymatrix{\H_\h \ar[r] \ar[d] & \H_{\D^\ast} \ar[r] \ar[d] & \H \ar[d] \cr
\h \ar[r]_q & \D^\ast \ar[r] & Y(N),}
$$
where $q(\tau) = e^{2\pi i \tau}$. The local system $\H_\h$ has a flat framing $\aa$ and $\bb$ as defined in \S\ref{sec:framed}. Define the flat vector bundle $\cH := \H \otimes \cO_{\M_{1,1}}$.\label{not:cH} The pullback $\cH_\h$ of $\cH$ to $\h$ has framings $\aa$, $\bb$,\label{not:ab} and also $\bX : \tau \mapsto \omega_\tau^\PD$, where $\omega_\tau \in H^0(E_\tau, \Omega^{1,0})$ is the Poincar\'e dual of the unique holomorphic differential taking the value 1 on $\aa$. It is straightforward to check $\omega_\tau = \tau\aa^\PD - \bb^\PD$, and hence, $\bX = \tau \aa - \bb$. The bundle $\cH_\h$ has natural connection $\nabla$ such that $\nabla \aa = \nabla \bb = 0$. Hence, $\nabla \bX = \aa \, d \tau$ and $\nabla = \aa \partial/\partial\bX \, d\tau$. Observe that $\aa$ and $\bX$ descend to a framing of $\cH_{\D^\ast} := \H_{\D^\ast} \otimes \cO_{\D^\ast}$. 

\subsection{The universal elliptic curve}

The group $\SL_2(\Z)$ acts on $\Z^2$ on the right by 
\begin{equation*}
    \label{eqn:Taction}
    \gamma : \begin{pmatrix}
    m & n
    \end{pmatrix} \mapsto \begin{pmatrix}
    m & n
\end{pmatrix} \gamma. 
\end{equation*}
The corresponding semidirect product $\SL_2(\Z) \ltimes \Z^2$ acts on $\C \times \h$ by 
\begin{equation}
    \label{eqn:ltimes}
    \left(\gamma,
    \begin{pmatrix}
    m & n
    \end{pmatrix}
    \right) : (z,\tau) \mapsto ((c\tau + d)^{-1}(z + m\tau + n), \gamma \tau).
\end{equation}
The (orbifold) quotient $(\SL_2(\Z) \ltimes \Z^2) \bbs (\C \times \h)$ is the {\em universal elliptic curve} $\E$. It is a fiber bundle over $\M_{1,1}$ where $E_\tau$ is the fiber over $\tau$. 

\subsection{Level structures} 

For any elliptic curve $E$, there is a natural isomorphism 
\begin{equation}
    \label{eqn:torsion}
    E[N] \cong H_1(E,\Z/N\Z). 
\end{equation}
A {\em level $N$ structure} on $E$ is an isomorphism 
$$
H_1(E,\Z/N\Z) \to (\Z/N\Z)^2
$$
where the intersection pairing on $H_1(E,\Z/N\Z)$ corresponds to the standard symplectic inner product on $(\Z/N\Z)^2$.

Denote the kernel of the natural homomorphism $\SL_2(\Z) \to \SL_2(\Z/N\Z)$ by $\G(N)$. It acts trivially on level $N$ structures via \eqref{eqn:framing}. Thus, the orbifold quotient $Y(N) := \G(N) \bbs \h$ is the moduli space of elliptic curves with level $N$ structure. 

A {\em congruence subgroup} $\G \subset \SL_2(\Z)$ of level $N$ is a subgroup containing $\G(N)$. The orbifold quotient $Y_\G := \G \bbs \h$ of the upper half plane $\h$ by a congruence subgroup $\G$ is called a {\em modular curve}. \label{not:YG} Of particular importance in \cite{thesis, gonch:sym} is the subgroup
$$
\G_1(N) := \left\{\begin{pmatrix} a & b \cr c & d \end{pmatrix} \equiv \begin{pmatrix} 1 & \ast \cr 0 & 1 \end{pmatrix} \bmod N \right\} \subset \SL_2(\Z).
$$
The action of $\G_1(N)$ on a framing $\{\aa,\bb\}$ of $H_1(E,\Z)$ stabilizes $\frac{1}{N}\aa \in H_1(E,\Z/N\Z)$, which via the natural isomorphism \eqref{eqn:torsion}, is equivalent to the choice of $N$-torsion point of $E$. Thus, the orbifold quotient $Y_1(N) := \G_1(N) \bbs \h$ is the moduli space of elliptic curves with a distinguished point of order $N$.  

Denote the pullback of the universal elliptic curve $\E \to \M_{1,1}$ to a modular curve $Y_\G$ by $\E_\G$.\label{not:EG} As in the introduction, we will abbreviate $\E_{\G(N)}$ by $\E_N$. We also denote the pullback of $\E$ to $\h$ by $\E_\h$. 

\subsection{Compactification}
\label{sec:Ecompact}

Let $\G \subset \SL_2(\Z)$ be a congruence subgroup. The quotient $\G \bs \P^1(\Q)$ is finite and its elements are called {\em cusps} of the modular curve $Y_\G$. For each cusp $P$, there exists $\alpha \in \SL_2(\Z)$ such that $\alpha \tau = i\infty$ under the action \eqref{eqn:modular}. The {\em width} of the cusp $P$ is the smallest positive integer $w_P$ such that 
$$
\alpha^{-1}\begin{pmatrix}
1 & w_P \cr 0 & 1
\end{pmatrix}
\alpha \in \G.
$$
The stabilizer of $P$ is given by
$$
\G_P = \G \cap \left\{\pm\alpha^{-1}\begin{pmatrix} 
1 & w_P\Z \cr 0 & 1 
\end{pmatrix} \alpha\right\}.
$$ 

Let $C_\G = \G \cap \{\pm I\}$. For each cusp $P$ of $Y_\G$, let $\D_P$ be an open disk of radius $e^{-2\pi/w_P}$. The compactification $X_\G := \overline{Y_\G}$ \label{not:XG} is obtained by gluing the quotient $C_\G \bbs \D_P$, where group $C_\G$ acts trivially on $\D_P$, to $Y_\G$ at each cusp along the maps 
$$
\xymatrix{C_\G \bbs \D_P & \ar[r] \G_P \bbs \{\tau \in \h \mid \Im(\tau) > 1\} \ar[l] & \G \bbs \h}.
$$
The left hand map sends $\tau \in \h$ to $e^{2\pi i\tau/w_P} \in \D_P$ and includes $C_\G$ into $\G_P$. We denote the compactifications of $Y_1(N)$ and $Y(N)$ by $X_1(N)$ and $X(N)$, respectively. 

The universal elliptic curve $\E_\G$ over $Y_\G$ has a natural smooth compactification $\overline{\E}_\G$ over $X_\G$ \label{not:EGbar}. The smooth orbi surface $\overline{\E}_\G$ is obtained by gluing in nodal curves above each cusp. If $P$ is a cusp of $Y_\G$ with width $w_P$, then the fiber of $\overline{\E}_\G$ above $P$ is a $w_P$-gon of $\P^1$'s obtained by taking $\P^1 \times (\Z/w_P\Z)$ and identifying $(\infty,n)$ with $(0,n+1)$ for each $n \in \Z/w_P\Z$. A full description is given in \cite{DR:nodal}. Above a cusp of width 1 (such as $\tau = i\infty$ in $X_1(N)$), the singular fiber is the nodal cubic $E_0$\label{not:nodal}. 

\subsection{Torsion sections} 
\label{torsion}

Denote the set of $N$-torsion sections of $\E_\G$ and $\E_\h$ by $\E_\G[N]$\label{not:EGN} and $\E_\h[N]$, respectively. Set $\E_\G' := \E_\G - \E_\G[N]$\label{not:EGminus} and $\E_\h' = \E_\h[N]$. If $\alpha \in \E_\G[N]$, it pulls back to a unique section $\tilde{\alpha} \in \E_\h[N]$. The lift $\tilde{\alpha}$ has {\em coordinates} $(x_\alpha,y_\alpha) \in (N^{-1}\Z/\Z)^2$ such that 
\begin{equation}
\label{eqn:lift}
    \tilde{\alpha} : \tau \longmapsto x_\alpha\tau +y_\alpha \bmod \Lambda_\tau.
\end{equation}
This defines a group isomorphism $\E_\h[N] \to (N^{-1}\Z/\Z)^2$, and the natural right action of $\SL_2(\Z)$ on $(\Z/N\Z)^2$ induces a right action on $\E_\h[N]$. Then 
$$
\E_\G[N] = \E_\h[N]^\G \cong \left((N^{-1}\Z/\Z)^2\right)^\G.
$$

The closure of a section $\alpha \in \E_\G[N]$ intersects the $\P^1$ and $E_0$ components of the singular fibers of $\overline{\E}_\G$ at an $N$th root of unity. For example, $Y_1(N)$ has two cusps: one under $\tau = i \infty$ has width 1 and the other under $\tau = 0$ has width $N$. The $N$ sections of $\E_{\G_1(N)}[N]$ intersect the nodal cubic over $\tau = i \infty$ at the roots of unity $\bmu_N \subset \Gm \subset E_0$, and they intersect the $N$-gon of projective lines over $\tau = 0$ at the identity in each copy of $\Gm \subset \P^1$. 

\subsection{Moduli with tangent vectors} 

Let $\cL_\G \to X_\G$ be the normal bundle of the identity section of the universal elliptic curve $\E_\G$. This is the dual of the Hodge bundle. The bundle $\cL_\G$ is the moduli space of elliptic curves $(E,P)$ with a $\G$-structure and distinguished tangent vector $\vv \in T_PE$ \cite[\S 5.4]{hain:ell}. Equivalently, the fiber of $\cL$ over an elliptic curve $(E,P)$ is the tangent space $T_PE$ of $E$ at its identity. 

Let $Y_{\G,\vec{1}}$ denote the $\C^\times$-bundle $\cL_\G^\times$ . It is the moduli space of elliptic curves with $\G$-structure and a {\em nonzero} tangent vector at its identity.

\section{Eisenstein series and Jacobi forms}

We briefly review some special functions which appear in the coefficients of the KZB connection. 

\subsection{Eisenstein series in full level $N$} Suppose that $m$ is a positive integer and that $\alpha$ is a section of $\E_N[N]$ over $Y(N)$ with coordinates $(x_\alpha,y_\alpha) \in (N^{-1}\Z/\Z)^2$. The Eisenstein series $G_{m,\alpha}(\tau)$\label{not:Gma} is given by
$$
G_{m,\alpha}(\tau) = \sum_{(k,\ell)\neq (0,0)} \frac{e^{2\pi i(ky_\alpha - \ell x_\alpha)}}{(k\tau + \ell)^m}.
$$
Note that the series does not depend on the choice of lift of $(x_\alpha, y_\alpha) \in (N^{-1}\Z/\Z)^2$ to $(N^{-1}\Z)^2$. The series converges almost uniformly on $\h$ when $m > 2$. When $m=2$, the series must be summed in a particular order if $\alpha = 0$. These Eisenstein series satisfy the twisted weight $m$ modularity property
\begin{equation}
    \label{eqn:eismodular}
    G_{m,\alpha}(\gamma\tau) = (c \tau + d)^m G_{m,\alpha\gamma}(\tau),
\end{equation}
for all $\gamma \in \SL_2(\Z)$. The right action of $\gamma$ on $\alpha$ is as in \eqref{eqn:Taction}. Thus, $G_{m,\alpha}(\tau)$ is modular with respect to a congruence subgroup $\G$ if and only if $\alpha \in \E_\G[N]$.

At the cusp under $\tau = i\infty$, the Eisenstein series $G_{m,\alpha}(\tau)$ takes the value
$$
\lim_{\tau \to i\infty} G_{m,\alpha}(\tau) = \sum_{\ell \neq 0} \frac{e^{-2\pi i\ell x_\alpha}}{\ell^m} = -\frac{(-2\pi i)^mB_m([x_\alpha])}{m!},
$$
where $B_m(x)$ is the $m$th Bernoulli polynomial and $[x_\alpha]$ is the unique rational number in $[0,1)$ such that $x_\alpha - [x_\alpha] \in \Z$.

\subsection{Eisenstein series with respect to $\G_1(N)$}

A section $\alpha \in \E_{\G_1(N)}[N]$ has coordinates $(0,y_\alpha)$. The corresponding Eisenstein series $G_{m,\alpha}(\tau)$ is modular with respect to $\G_1(N)$. In this case, we will typically write $G_{m,\alpha}(\tau)$ as $G_{m,\zeta}(\tau)$\label{not:Gmz} where $\zeta$ is the $N$th root of unity $e^{2\pi i y_\alpha}$. The defining series becomes
$$
G_{m,\zeta}(\tau) = \sum_{(k,\ell) \neq (0,0)} \frac{\zeta^k}{(k\tau + \ell)^m}.
$$
If $q = e^{2\pi i \tau}$, then $G_{m,\zeta}(\tau)$ has Fourier expansion $\sum_{n = 0}^\infty a_n q^n$, where 
$$
a_0 = (1 + (-1)^m)\zeta(m)
$$
and 
\begin{equation}
    \label{eqn:eisfourier}
    a_n = \frac{(-2\pi i)^m}{N^m(m-1)!}\left(\sum_{d|n} \sgn(d) d^{m-1} +  \sum_{\substack{(k,\ell) \in (\Z/N)^2 \\ (k,\ell) \neq (0,0)}} \zeta^k \sum_{\substack{d|nN \\ Nn/d \equiv k}} \sgn(d)d^{m-1}e^{2\pi i \ell d/N}\right)
\end{equation}
when $n \ge 1$. The corresponding normalized Eisenstein series is
$$
\bG_{m,\zeta}(\tau) = \frac{(m - 1)!}{(2\pi i)^m}\frac{1}{\zetabar + (-1)^m\zeta}G_{m,\zeta}(\tau) = -\frac{B_m}{m(\zetabar + (-1)^m\zeta)} + q + \cdots.
$$
Except when $m = 2$ and $\zeta = 1$, the functions $G_{m,\zeta}$ and $\bG_{m,\zeta}$ satisfy the modularity property
$$
G_{m,\zeta}(\gamma\tau) = (c\tau + d)^m G_{m,\zeta}(\tau) 
$$
for all $\gamma \in \G_1(N)$. Complex conjugation induces the relations 
$$
G_{m,\zetabar}(\tau) = (-1)^m G_{m,\zeta}(\tau) \quad \text{and} \quad \bG_{m,\zetabar}(\tau) = \bG_{m,\zeta}(\tau). 
$$
Finally, the values of $G_{m,\zeta}$ at the cusps are given by 
$$
G_{m,\zeta}(\tau)\big|_{q =0} = (1 + (-1)^m)\zeta(m)
$$
and 
$$
(c\tau + d)^{-m} G_{m,\zeta}(\gamma\tau)\big|_{q = 0} = \sum_{k \neq 0} \frac{\zeta^{- ck}}{k^m}.
$$
for all $\gamma \in \SL_2(\Z)$.

\subsection{Jacobi forms}

Other functions appearing in the KZB connection are derived from Zagier's Jacobi form \cite{zagier}
$$
F^\Zag_\tau(u,v) = \frac{\theta'(0|\tau)\theta(u+v|\tau)}{\theta(u|\tau)\theta(v|\tau)},
$$
where $\theta(z|\tau)$ is the classical theta function
$$
\theta(z|\tau) = \sum_{n \in \mathbb{Z}} (-1)^n e^{i \pi \tau(n + 1/2)^2}e^{z(n + 1/2)}.
$$
Levin and Racinet \cite{LR} express the KZB connection in terms of the function $F^\LR(x,z,\tau)$, which is related to Zagier's function by 
$$
F^\LR(x,z,\tau) = 2\pi i F^\Zag(2\pi i x,2\pi iz, \tau).
$$
Calaque--Enriquez--Etinghof \cite{CEE}, Gonzalez \cite{gonza}, and Calaque--Gonzalez \cite{CG} implicitly use the same Jacobi form as Levin--Racinet.\footnote{The function $k(x,z|\tau)$ in \cite{CEE,gonza,CG} can be expressed as $F^\LR(x,z,\tau) - \frac{1}{x}$.} We shall use the Jacobi form\footnote{This normalization is used implicitly in \cite{rome} and is best suited for viewing the KZB connection as the de~Rham realization of a variation of MHS.}
\begin{equation}
    \label{eqn:F}
    F(x,z,\tau) := F^\LR(x/(2\pi i),z,\tau) = 2\pi iF^\Zag(x,2\pi iz,\tau).
\end{equation}
Zagier \cite[\S 3]{zagier} computes the Fourier expansion of $F$ to be 
\begin{equation}
        \label{eqn:Ffourier}
         F(x,z,\tau) = \pi i\left(\coth(x/2) + \coth(\pi i z)\right) + 4\pi i \sum_{n = 1}^\infty \sum_{d|n} \sinh\left(dx + \frac{2\pi i nz}{d}\right)q^n,
\end{equation}
where $q = e^{2\pi i \tau}$. It follows that $F(x,z,\tau)$ has simple poles along $x = 0$ and $z = 0$ with residue $2\pi i$ and 1, respectively
\begin{equation}
    \label{eqn:Fres}
    F(x,z,\tau) = \frac{2\pi i}{x} + \frac{1}{z} + \text{holomorphic terms}.
\end{equation}

\section{Unipotent completion}
\label{sec:unipotent}

In this section we review some basic facts about unipotent completion. More details can be found in \cite{knudson}.

Let $\G$ be a discrete group and $F$ a field of characteristic zero. The unipotent completion of $\G$ over F is a pro-unipotent group $\cU$ over $F$ together with a homomorphism $\phi: \G \to \cU(F)$ with the following universal property: if $\psi : \G \to U(F)$ is a homomorphism from $\G$ into the $F$-points of a unipotent $F$-group $U$, then there is a unique homomorphism $\cU \to U$ of $F$-groups such that the diagram
$$
\xymatrix{\G \ar[dr]_\psi \ar[r]^\phi & \cU(F) \ar[d] \cr & U(F) }
$$
commutes.

When $H_1(\G,F)$ is finite dimensional, there is a concrete construction of $\U$ that goes back to Quillen \cite{quillen}. Denote the group algebra of $\G$ over $F$ by $F\G$. It is a Hopf algebra with coproduct that takes each $\gamma \in \G$ to $\gamma \otimes \gamma$. The standard augmentation $\e : F\G \to F$ takes each $\gamma \in \G$ to 1. The augmentation ideal $I$ is its kernel. The $I$-adic completion of $F\G$ is the complete augmented Hopf algebra
$$
F\Gamma^\wedge = \varprojlim_n F\Gamma/I^n.
$$ 
The coproduct on $F\G$ induces an augmentation preserving $F$-algebra homomorphism $\Delta : F\Gamma^\wedge  \to F\Gamma^\wedge \hat{\otimes} F\Gamma^\wedge$.

Define the group-like elements 
$$
\mathcal{P}(F) = \{x \in F\Gamma^\wedge \mid \epsilon(x) = 1 \text{ and } \Delta x = x\otimes x\}
$$ 
and the primitive elements 
$$
\mathfrak{p}(F) = \{x \in F\Gamma^\wedge \mid \Delta x = x \otimes 1 + 1 \otimes x\}.
$$ 
Then $\cP(F) \subset 1 + I^\wedge$ forms a group and $\p(F) \subset I^\wedge$ is its Lie algebra. The inclusion $\G \to F\G$ induces a natural homomorphism $\G \to \cP(F)$. Exponentiation $\exp : \mathfrak{p}(F) \to \mathcal{P}(F)$ is a group isomorphism, where the group structure of $\mathfrak{p}(F)$ is determined by the Baker--Campbell--Hausdorff formula \cite{serre}.

Define the continuous dual
$$
(F\Gamma)^\vee := \varinjlim_n \Hom_F(F\Gamma/I^n,F).
$$
When $H_1(\G,F)$ is finite dimensional, this is a commutative Hopf algebra. Thus, $\Spec (F\Gamma)^\vee$ is a pro-unipotent affine $F$-group. There is a homomorphism
$$
\G \to \Spec (F\Gamma)^\vee
$$
defined by $\gamma \mapsto \{\phi \mapsto \phi(\gamma)\}$. By the universal mapping property, this induces a map
$$
\U \to \Spec (F\Gamma)^\vee.
$$
\begin{prop}
If $H_1(\G,F)$ is finite dimensional, then this is an isomorphism of pro-unipotent $F$-groups. It induces an isomorphism
$$
\U(F) \to \mathcal{P}(F)
$$
The Lie algebra of $\U$ is naturally isomorphic to $\p$.
\end{prop}

\begin{example}
\label{ex:free}
Let $\G = \langle \sigma_1,\ldots,\sigma_n\rangle$ be a free group on $n$ generators. Then $\p(F)$ is the completed free Lie algebra $\bL(\ee_1,\ldots,\ee_n)^\wedge$ \label{not:free}over $F$ generated by $\{\ee_1,\ldots,\ee_n\}$. The map $\G \to \cP(F)$ takes $\sigma_j$ to $\exp\ee_j$.
\end{example}

Since KZB is concerned with unipotent fundamental groups, we will abbreviate the unipotent completion of $\pi_1(X,x)$ as $\piun(X,x)$\label{not:piun} and its Lie algebra as $\p(X,x)$.\label{not:p} Unless specified otherwise, we always take the coefficient field to be $\Q$.

\section{The holomorphic vector bundle $\bP_N$}
\label{sec:KZBinfo}

In this section we construct the holomorphic vector bundle $\bP_N \to \E_N$ on which the KZB connection is defined. As in the $N = 1$ case \cite{CEE,LR}, we will define $\bP_N$ as a quotient of the trivial bundle $\h\times\C \times \p_N \to \h\times\C$ over the universal cover $\h \times \C$ of $\E_N$ by an action of $\SL_2(\Z)\ltimes \Z^2$ which acts via a suitable factor of automorphy.

Note that Calaque and Gonzalez \cite{CG} define their connection over $\M_{1,n+1}(N)$, the moduli space of $(n+1)$-pointed elliptic curves with a level $N$ structure. We consider only the $n=1$ case, where $\M_{1,2}(N) = \E_N'$, as that is all that is needed for our applications.

\subsection{Factors of automorphy}

Suppose a discrete group $G$ acts on a topological space $X$ on the left and that $V$ is a left $G$-module. A {\em factor of automorphy} is a function $M : G \times X \to \Aut V$ written as $(g,x) \mapsto M_g(x)$ such that $g$ acts on $X \times V$ via
$$
g : (x,v) \mapsto (g \cdot x,M_g(x)v). 
$$
This requires 
\begin{equation}
    \label{eqn:Mg}
    M_{gh}(x) = M_g(hx)M_h(x)
\end{equation}
for all $g, h \in G$ and $x \in X$. A section $s$ of the bundle $X \times V \to X$ is $G$-equivariant and descends to a section of $G\bs (X \times V) \to G\bs X$ if $s(g\cdot x) = M_g(x)s(x)$ . 

\subsection{Factors of automorphy of $\bP_N$}
\label{sec:PN}

The factors of automorphy of $\bP_N \to \E_N$
$$
\widetilde{M}_{\gamma,(m,n)}(\tau,z) : (\SL_2(\Z) \ltimes \Z^2) \times (\h \times \C) \to \Aut \p_N
$$
are defined by the formula
\begin{equation}
    \label{eqn:automorphy}
    \widetilde{M}_{\gamma,(m,n)}(\tau,z) = \left\{
        \begin{array}{ll}
            M_{\gamma}(\tau) \circ \exp(cz\bX/(c\tau + d)) & \gamma \in \G(N) \cr
            \exp(-m\bX) & (m,n) \in \Z^2,
        \end{array}
    \right.
\end{equation}
where 
\begin{equation}
    \label{eqn:M}
    M_{\gamma}(\tau) : \left\{\begin{array}{ccl}
    \bX &\longmapsto &(c\tau + d)\bX  \\
    \bY &\longmapsto &(c\tau + d)^{-1} \bY + \frac{c}{2\pi i} \bX \\
    \bt_\alpha &\longmapsto &\bt_{\alpha \gamma^{-1}}
\end{array}\right.
\end{equation}

\begin{remark}
The factor of automorphy $M_\gamma(\tau)$ is the same as \eqref{eqn:framing} after the change of basis $\bX = \tau \aa - \bb$ and $\bY = \frac{1}{2\pi i}\aa$. Thus, the quotient of $\bP_N$ by the subalgebra generated by $\{\bt_\alpha\}$ is isomorphic to the pullback of $\H$ to $\E_N$.
\end{remark}
\begin{prop}
The factors of automorphy $\widetilde{M}_{\gamma,(m,n)}(z,\tau)$ satisfy \eqref{eqn:Mg}.
\end{prop}
\begin{proof}
This follows from \eqref{eqn:ltimes}, \eqref{eqn:automorphy}, and \eqref{eqn:M}. See \cite[\S 6]{hain:kzb}.
\end{proof}

Thus, there is a well-defined left action of $\G(N) \ltimes \Z^2$ on the trivial bundle $\h \times \C \times \p_N \to \h \times \C$ given by 
$$
(\gamma,(m,n)) : (\tau,z,\v) \mapsto \left(\gamma\tau, \frac{z + m\tau + n}{c\tau + d}, \widetilde{M}_{\gamma,(m,n)}(\tau,z)(\v)\right).
$$
The bundle $\bP_N \to \E_N$ is the quotient of $\h \times \C \times \p_N \to \h \times \C$ by this action. Only small modifications are necessary to construct the analogous bundle over $\E_\G$ for any congruence subgroup $\G$ (see \S\ref{sec:other}).

\section{The connection form}

We can finally define the KZB connection. The formulas in \S\ref{sec:der} through \S\ref{sec:kzb} are adaptations of \cite{CG,gonza}.

\subsection{Derivations}
\label{sec:der}

The connection form takes values in $\Der \p_N$. The relevant derivations are indexed by integers $m \ge 0$ and $N$-torsion sections $\alpha \in \E_N[N]$.
$$
\delta_{m,\alpha} : \left\{ 
\begin{array}{ccl}
    \bX & \longmapsto & 0 \cr
    \bY & \longmapsto & \displaystyle \sum_{j + k = m  - 1} \sum_\beta (-1)^j [\ad_{\bX}^j(\bt_\beta),\ad_{\bX}^k (\bt_{\alpha + \beta})] \cr
    \bt_\beta & \longmapsto & [\bt_\beta, \ad^m_{\bX}(\bt_{\beta + \alpha}) + (-1)^m\ad^m_{\bX}(\bt_{\beta-\alpha})].
\end{array}
\right.
$$
It is straightforward to verify
$$
\delta_{m,\alpha}([\bX,\bY]) = \sum_\beta \delta_{m,\alpha}(\bt_\beta),
$$ 
and therefore $\delta_{m,\alpha}$ is well-defined on $\p_N$.
For $m \geq 0$, set
\begin{equation}
\label{eqn:ep}
\epsilon_{m + 2,\alpha} = \delta_{m,\alpha} + \ad_{\bX^m}(\bt_\alpha + (-1)^m\bt_{-\alpha}).
\end{equation}
The identities
$$
\e_{m+2,\alpha}(\bt_0) = 0 \quad \text{and} \quad \e_{m,\alpha} = (-1)^m\e_{m,-\alpha}
$$
are useful in calculations.
\begin{remark}
\label{rem:esp}
The derivations $\e_{m,\alpha}$ are {\em special}. That is, they annihilate $\bt_0$ and for every $\beta \in \E_N[N]$ satisfy
$$
\e_{m,\alpha} : \bt_\beta \mapsto [u_\beta,\bt_\beta]
$$
for some $u_\beta \in \p_N$.
\end{remark}

\subsection{Coefficients}
\label{sec:coeff}

The coefficients of the connection form are constructed from Jacobi forms and indexed by $\alpha \in \E_N[N]$. Define
$$
h_\alpha(x,z|\tau) := e^{-x_\alpha x} F(x,z - \tilde{\alpha},\tau) - \frac{2\pi i}{x}, 
$$
where $F$ is the Jacobi form from \eqref{eqn:F}. The function $h_\alpha$ is independent of choice of lift $\tilde{\alpha}$ due to the elliptic property \cite[\S3]{zagier} 
$$
F(x,z + m\tau + n,\tau) = e^{-mx}F(x,z,\tau).
$$ 
Let $g_\alpha(x,z|\tau)$ denote the partial derivatives
$$
g_\alpha(x,z|\tau) := \frac{\partial}{\partial x}h_\alpha(x,z|\tau).
$$

The Jacobi form $F$ also satisfies a modularity property \cite[\S3]{zagier}
$$
F\left(\frac{x}{c\tau + d},\frac{z}{c\tau + d},\gamma\tau\right) = (c\tau+d)\exp\left(\frac{czx}{c\tau + d}\right)F(x,z,\tau)
$$
for all $\gamma \in \SL_2(\Z)$. This induces the relations
\begin{equation}
\label{eqn:hinv}
    \begin{aligned}
    h_\alpha(x,z|\tau + 1) &= h_{\alpha T}(x,z|\tau) \cr 
    h_\alpha\left(x,\frac{z}{\tau} \middle| -\frac{1}{\tau}\right) &= \tau e^{zx} h_{\alpha S}(\tau x,z|\tau) + 2\pi i \frac{e^{zx} - 1}{x},
    \end{aligned}
\end{equation}
where $T = \mat{1 & 1 \cr 0 & 1}$ and $S = \mat{0 & -1 \cr 1 & 0}$\label{not:ST} are generators of $\SL_2(\Z)$. Differentiating yields 
\begin{equation}
\label{eqn:ginv}
\begin{aligned}
    g_\alpha(x,z|\tau + 1) &= g_{\alpha T} (x,z|\tau) \cr 
    g_\alpha \left(x,\frac{z}{\tau} \middle| -\frac{1}{\tau}\right) &= \tau z e^{zx}h_{\alpha S}(\tau x,z|\tau) \cr 
    & \qquad + \tau^2 e^{zx} g_{\alpha S}(\tau x,z|\tau) \cr 
    & \qquad + 2\pi i \frac{zxe^{zx} - e^{zx} + 1}{x^2}.
\end{aligned}
\end{equation}

Define $A_{m,\alpha}(\tau)$ to be the coefficients of the Taylor expansion of $g_\alpha(x,0|\tau)$ with respect to $x$:
$$
g_\alpha(x,0|\tau) = \sum_{m = 0}^{\infty}A_{m,\alpha}(\tau)x^m.
$$
It follows directly from \eqref{eqn:ginv} that these functions $A_{m,\alpha}(\tau)$ satisfy the weight $m+2$ modularity equation 
\begin{equation}
    \label{eqn:Amodular}
    A_{m,\alpha}(\gamma \tau) = (c\tau + d)^{m+2}A_{m,\alpha\gamma}(\tau)
\end{equation}
for all $\gamma \in \SL_2(\Z)$. 

\subsection{Formula}
\label{sec:kzb}

We can now write down an explicit formula for the level $N$ KZB connection. It is given by 
$$
\nabla_{\KZB_N} = d + \Omega_N,
$$
where $\Omega_N$ is a 1-form on $\h \times \C$ with logarithmic singularities along $\Lambda_N$, the preimage of $\E_N[N]$ in $\h \times \C$. It takes values in $\Der \p_N$, which  acts on the left of $\p_N$. 

Define the following 1-forms.  
$$
\psi = \frac{1}{2} \sum_{\substack{m \geq 0 \\ \alpha \in (N^{-1}\Z/\Z)^2}} A_{m,\alpha}(\tau)\delta_{m,\alpha}\, d\tau
$$

$$
\nu_1 = \sum_{\alpha \in (N^{-1}\Z/\Z)^2} g_\alpha\left(\bX,z\middle|\tau\right) \cdot \bt_\alpha \, d\tau
$$

$$
\nu_2 = \left(2\pi i\bY + \sum_{\alpha \in (N^{-1}\Z/\Z)^2} h_\alpha\left(\bX,z\middle|\tau\right) \cdot \bt_\alpha\right) \, dz 
$$
Then the KZB connection form is
$$
\Omega_N = 2\pi i \bY \frac{\partial}{\partial \bX} \, d\tau + \psi + \nu_1 + \nu_2.
$$
It follows from \eqref{eqn:Fres} that $h_\alpha(\bX,z|\tau)$ has simple poles along the preimage of the section $\alpha \in \E_N[N]$ in $\h \times \C$. We also know from \eqref{eqn:Ffourier} that $F(x,z,\tau)$ is holomorphic with respect to $q = e^{2\pi i \tau}$ and non-vanishing at $q = 0$. Thus, $\Omega_N$ has logarithmic singularities along $\E_N[N]$ and the singular fibers of $\E_N$ as claimed. 
\begin{prop}
\label{prop:invariance}
The KZB connection is invariant with respect to the action of $\G(N) \ltimes \Z^2$ on the trivial bundle $\h \times \C \times \p_N \to \h \times \C$. Thus, $\nabla_{\KZB_N}$ is a well-defined connection on the bundle $\bP_N \to \E_N$.
\end{prop}
\begin{proof}
The connection is invariant if 
$$
\nabla_{\KZB_N}(\widetilde{M}_gs) = \widetilde{M}_g(\nabla_{\KZB_N} s)
$$ 
for all $g \in \G(N) \ltimes \Z^2$ and sections $s : \h \times \C \to \p_N$. This is equivalent to
$$
g^\ast \Omega_N = \widetilde{M}_g\Omega_N \widetilde{M}_g^{-1} - (d\widetilde{M}_g)\widetilde{M}_g^{-1}
$$
for all $g \in \G(N) \ltimes \Z^2$. Verifying this equation involves a lengthy computation requiring \eqref{eqn:hinv} and \eqref{eqn:ginv}. See Appendix \ref{sec:invariance}.
\end{proof}

\begin{prop}
The KZB connection is flat, i.e. the curvature form $d\Omega_N + \Omega_N \wedge \Omega_N$ vanishes.
\end{prop}
\begin{proof}
See \cite[Prop. 3.9]{CG}. 
\end{proof}
This generalizes the proof in \cite[Prop 1.2]{CEE} of flatness in the $N = 1$ case. The $N = 1$ case was also independently proven in \cite[Prop. 3.2.2]{LR} (see also \cite[\S 9.4--5]{hain:kzb}).

\section{Monodromy and iterated integrals}

In section \S\ref{sec:rigidity}, we prove our first main result: The KZB connection on $\bP_N$ is isomorphic to the flat vector bundle $\bP_N^\topo\otimes_\Q \cO_{\E'_N}$. We will make this argument by identifying the fibers over a single point and proving their monodromy representations are compatible. In this section, we review the monodromy formula and use it to construct an isomorphism by which we can identify these fibers.

\subsection{Monodromy of trivial bundles with connection}

Suppose a discrete group $G$ acts on a topological space $X$ on the left and that $V$ is a $G$-module with factor of automorphy $M : G \times X \to \Aut V$. Fix a base point $x_o \in G\bs X$ and lift $\tilde{x}_o \in X$. This determines a surjection $\rho : \pi_1(G\bs X,x_o) \to G$.
\begin{prop}
\label{prop:mono}
Suppose that $\nabla$ is a $G$-invariant connection on $X \times V \to X$. Denote the inverse of monodromy\footnote{Our paths convention makes monodromy an anti-homomorphism. To get a homomorphism, we use its inverse instead.} of $G\bs (V \times X) \to G\bs X$ by $\Theta_{x_o} : \pi_1(G\bs X,x_o) \to \Aut V$. It is a homomorphism given by
$$
\Theta_{x_o}(\beta) = T(\tilde{\beta})^{-1} \circ M_{\rho(\beta)}(\tilde{x}_o),
$$
where $\tilde{\beta}$ is a lift of $\beta$ to $X$ based at $\tilde{x}_o$ and $T(\tilde{\beta})$ is the parallel transport of $X \times V \to X$ with respect to $\nabla$ along $\tilde{\beta}$.
\end{prop}
\begin{proof}
See \cite[\S5.2]{hain:kzb}.
\end{proof}
To make $\Theta_{x_o}$ explicit, we introduce iterated integrals. Suppose $\omega_1,\ldots,\omega_m$ are smooth 1-forms on $X$. Define the {\em iterated integral} 
$$
\int_{\tilde{\beta}} \omega_1 \cdots \omega_m := \int\limits_{0 \le t_1 \le \cdots \le t_m \le 1} f_1(t_1) \cdots f_m(t_m) \, dt_1 \cdots dt_m,
$$
where $\tilde{\beta}^\ast\omega_k = f_k(t) \, dt$. 
\begin{prop}
\label{prop:chen}
Suppose $\tilde{\beta} : [0,1] \to X$ is a piecewise smooth path. The {\em inverse parallel transport} of $\nabla = d + \omega$ on $V \times X \to X$ along $\tilde{\beta}$ is given by 
\begin{equation*}
    \label{eqn:chenT}
    T(\tilde{\beta})^{-1} = 1 + \int_{\tilde{\beta}} \omega + \int_{\tilde{\beta}} \omega\omega + \int_{\tilde{\beta}} \omega \omega \omega + \cdots.
\end{equation*}
\end{prop}
\begin{proof}
This is Chen's formula \cite[\S 3]{chen}. See \cite[Lemma 2.5]{hain:bowdoin} for a direct proof.
\end{proof}
If the connection $\nabla$ is flat (e.g. KZB), then $T(\tilde{\beta})^{-1}$ depends only on the homotopy class of $\tilde{\beta}$. 

\subsection{Monodromy in a single fiber of $\bP_N$}

Let us apply the formulas in Propositions \ref{prop:mono} and \ref{prop:chen} to the KZB connection. Fix an elliptic curve $E_\tau = \C/\Lambda_\tau$ and a base point $x \in E_\tau'$. Parallel transport of the KZB connection restricted to the fiber $E_\tau'$ of $\E'$ induces an inverse monodromy representation 
$$
\Theta_\tau : \pi_1(E_\tau',x) \to \Aut\p_N.
$$

\begin{prop}
\label{prop:fiber}
The inverse of monodromy of the KZB connection restricted to the fiber $E_\tau'$ induces an isomorphism of completed free Lie algebras
$$
\psi_\tau : \p(E_\tau',x) \otimes_\Q \C \to \p_N.
$$
\end{prop}
\begin{proof}
Fix a loop $\beta \in \pi_1(E_\tau',x)$ whose homology class $[\beta] \in H_1(E_\tau',\Z)$ has coordinates $m,n,c_\alpha \in \Z$ so that 
$$
[\beta] = n\aa + m\bb + \sum_\alpha c_{\alpha}\bt_\alpha \in H_1(E_\tau',\Z),
$$
where $\aa$ and $\bb$ are the framing of $H_1(E_\tau,\Z)$ defined in \S\ref{sec:framed} and each $\bt_{\alpha}$ denotes a small positively oriented loop about $\alpha \in E_\tau[N]$. Thus, $\beta$ lifts to a path $\tilde{\beta}$ in $\C$ from 0 to $m\tau + n$. We will now apply Proposition \ref{prop:mono} to compute $\Theta_\tau(\beta)$.

Since $\tau$ is fixed, we need only consider the $dz$ component of $\Omega_N$. Thus, by Proposition \ref{prop:chen} and the factors of automorphy \eqref{eqn:automorphy},
\begin{equation}
\label{eqn:fibermono}
\begin{aligned}
    \Theta_\tau(\beta) &= \left(1 + \int_{\tilde{\beta}} \nu_2 + \int_{\tilde{\beta}} \nu_2\nu_2 + \cdots \right)e^{-m\ad\bX} \cr
    &= \exp\left(\int_{\tilde{\beta}} \nu_2\right)e^{-m \ad\bX} \cr
    &= \Ad \left(\exp\left(\int_{\tilde{\beta}} \nu_2\right)e^{-m \bX}\right).
\end{aligned}
\end{equation}
Thus, the image of $\Theta_\tau$ is contained in the inner automorphisms of $\p_N$, which defines a map $\pi_1(E_\tau',x) \to \exp \p_N$. This induces a map from the unipotent completion $\piun(E_\tau',x) \to \exp \p_N$ and a map of Lie algebras $\psi_\tau : \p(E_\tau',x) \to \p_N$ given by 
$$
\psi_\tau : \log \beta \longmapsto \log \left(\exp\left(\int_{\tilde{\beta}} \nu_2\right)e^{-m \bX}\right).
$$
It remains to show that this map is an isomorphism after tensoring with $\C$. Recall from \S\ref{sec:coeff} that
$$
\nu_2 \equiv \left(2\pi i\bY + \sum_{\alpha} \frac{\bt_\alpha}{z - \tilde{\alpha}}\right) dz \bmod [\p_N,\p_N].
$$
Thus, 
$$
\int_{\tilde{\beta}} \nu_2 \equiv (m \tau + n)\aa + \sum_\alpha c_\alpha \bt_\alpha \bmod [\p_N,\p_N],
$$ 
and by the Baker--Campbell--Hausdorff formula,
\begin{align*}
\psi_\tau(\log \beta) &\equiv \log \left(\exp\left(\int_{\tilde{\beta}} \nu_2\right)e^{-m \bX}\right) \cr
&\equiv \left(\int_{\tilde{\beta}} \nu_2\right) - m\bX \cr
&\equiv n\aa + m\bb + \sum_\alpha c_\alpha \bt_\alpha  \cr
&\equiv [\beta] \bmod [\p_N,\p_N].
\end{align*}
Therefore, the induced map on homology $(\psi_\tau)_\ast : H_1(\p(E_\tau',x)) \otimes_\Q \C \to H_1(\p_N)$ is an isomorphism. The result follows from the fact that both $\p(E_\tau',x)$ and $\p_N$ are free.
\end{proof}

\section{Rigidity}
\label{sec:rigidity}

We are now ready to prove the KZB connection and $\bP_N^\topo \otimes \cO_{\E_N'}$ are isomorphic as flat vector bundles. Identify the fibers of $\bP^\topo_N \otimes_\Q \C$ and $\bP_N$ above the base point $[E_\tau,x]$ by the isomorphism $\psi_\tau$ in Proposition \ref{prop:fiber}. This also induces an identification of automorphism groups via
$$
\Ad \psi^{-1}_\tau : \Aut \p_N \to \Aut \piun(E_\tau',x).
$$
Let $G$ denote the fundamental group $\pi_1(\E_N',[E_\tau,x])$. There are monodromy representations 
$$
\rho^\KZB : G \to \Aut \p_N \quad \text{and} \quad \rho^\topo : G \to \Aut \piun(E_\tau',x) 
$$
of $\bP_N$ and $\bP_N^\topo$, respectively. To complete the argument, we must prove the diagram
\begin{equation}
    \label{eqn:rigidity}
    \begin{gathered}
        \xymatrix{
        & \Aut \p_N \ar[dd]^{\Ad\psi^{-1}_\tau} \cr
        G \ar[ru]^{\rho^\KZB} \ar[rd]_{\rho^\topo} & \cr
        & \Aut \p(E_\tau',x) 
        }
    \end{gathered}
\end{equation}
commutes.

We apply a general lemma of Hain. Let $N$ be a normal subgroup of a discrete group $\G$. Let $\mathcal{N}$ denote the unipotent completion of $N$. The conjugation homomorphism $\G \to \Aut N$ induces a map $\phi : G \to \Aut \mathcal{N}$. Restriction of $\phi$ to $N$ maps $n \in N$ to $\iota_{\theta(n)}$, where $\theta : N \to \mathcal{N}$ is inclusion and $\iota_u \in \Aut \mathcal{N}$ is conjugation by $u \in \mathcal{N}$.  
\begin{lemma}[{{\cite[Lemma 14.1]{hain:kzb}}}]
If $\mathcal{N}$ has trivial center, then $\phi$ is the unique homomorphism $\G \to \Aut \mathcal{N}$ whose restriction to $N$ is $n \mapsto \iota_{\theta(n)}$.
\end{lemma}
\begin{theorem}
The KZB connection over $\E_N'$ is naturally isomorphic to the flat vector bundle $\bP^\topo_N\otimes_\Q \cO_{\E_N'}$.
\end{theorem}
\begin{proof}
We know $\piun(E',x)$ is free and thus has trivial center. The monodromy representation $\rho^\topo$ is induced by conjugation. Set $\G$ to be $G = \pi_1(\E_N',[E_\tau,x])$, $N = \pi_1(E_\tau',x)$, and $\phi = \rho^\topo$. Applying the lemma, it suffices to show the diagram \eqref{eqn:rigidity} commutes when restricting to $N = \pi_1(E_\tau',x)$. 

Set $\varphi = \exp \circ \psi_\tau \circ \log$. By \eqref{eqn:fibermono},
$$
\rho^\KZB(\beta) = \Ad (\exp \psi (\log \beta)) = \Ad(\varphi(\beta))
$$
for all $\beta \in \pi_1(E_\tau',x)$. Then 
\begin{align*}
    ((\Ad \psi^{-1}_\tau) \cdot \rho^\KZB)(\beta) &= \psi^{-1}_\tau \circ \Ad(\varphi(\beta)) \circ \psi_\tau \cr
    &= \psi^{-1}_\tau \circ \exp(\ad \psi_\tau(\log \beta)) \circ \psi_\tau \cr
    &= \exp (\ad \log \beta) \cr
    &= \Ad \beta \cr
    &= \rho^\topo(\beta).
\end{align*}
\end{proof}




\section{Restrictions and residues}

In \S\ref{sec:avmhsP} we will prove $\bP_N$, endowed with its natural Hodge and weight filtrations, is an admissible variation of MHS and that the isomorphism of Lie algebras $\psi_\tau : \p(E_\tau',x) \otimes_\Q \C \to \p_N$ in Proposition \ref{prop:fiber} is in fact an isomorphism of MHS. A preliminary step is to compute the residues of the connection along the boundary divisor of $\E_N'$ in $\overline{\E}_N$. We devote this section to this task. 

\subsection{Restriction to torsion sections}

We will first compute the restriction of the KZB connection to a first-order neighborhood of the identity section of $\E_N \to Y(N)$. This is equivalent to restricting the connection to the moduli $Y_{\G(N),\vec{1}}$. Concretely, the restriction is given by 
$$
G(z,\tau)\, d\tau + H(z,\tau) \, \frac{dz}{z} \mapsto G(0,\tau) \, d\tau + H(0,\tau)\, \frac{dz}{z}.
$$
\begin{prop}
\label{prop:coeff}
The Taylor coefficients $A_{m,\alpha}(\tau)$ are level $N$ Eisenstein series
\begin{equation*}
    \label{eqn:coeff}
    A_{m,\alpha}(\tau) = -\frac{(m+1)}{(2\pi i)^{m+1}} G_{m+2,\alpha}(\tau).
\end{equation*}
\end{prop}
\begin{proof}
One can use \eqref{eqn:Ffourier} and \eqref{eqn:eisfourier} to show the functions above have identical Fourier expansions when $x_\alpha = 0$ and $y_\alpha = 0$ or 1. The result follows from the modularity properties of $A_{m,\alpha}(\tau)$ and $G_{m+2,\alpha}(\tau)$ in \eqref{eqn:Amodular} and \eqref{eqn:eismodular}, respectively.
\end{proof}
\begin{cor}
The restriction of the level $N$ KZB connection form to $Y_{\G(N),\vec{1}}$ is
$$
\Omega_N' = \left(2\pi i \bY \frac{\partial}{\partial \bX} - \frac{1}{2} \sum_{\substack{m \ge 2 \\ \alpha \in (N^{-1}\Z/\Z)^2}} \frac{(m - 1)}{(2\pi i)^{m-1}} G_{m,\alpha}(\tau) \e_{m,\alpha} \right)\, d\tau + \bt_0 \, \frac{dz}{z}.
$$
\end{cor}
In particular, the residue of $\Omega_N'$ along the zero section is $L_{z = 0} = \bt_0$.\footnote{The standard notation for nilpotent residues is $N$. Since $N$ already denotes level, we use $L$ instead.} Similarly, along any lift of a torsion section $z = \tilde{\alpha}$, the residue of $\Omega_N'$ is 
$$
L_{z,\tilde{\alpha}} := \Res_{z = \tilde{\alpha}} \Omega_N' = e^{-x_\alpha\bX} \cdot \bt_\alpha.
$$

\subsection{Restriction to singular fibers}
\label{sec:res}

Let $q_N = e^{2\pi i \tau/N}$. Then $2 \pi i \, d\tau = N \, \frac{dq_N}{q_N}$. Recall from \S\ref{sec:Ecompact} that the singular fiber of $\overline{\E}_N$ over $q_N = 0$ is an $N$-gon of $\P^1$'s. The restriction of $\Omega_N$ to a first order neighborhood of identity component\footnote{The restriction to the other components can be similarly computed after pulling the connection back along the translation $z \mapsto z - \tilde{\alpha}$, where $\tilde{\alpha}$ is the lift of an torsion section of $\E_N$ to $\h \times \C$. } of this fiber is given by
$$
G(z,\tau) \frac{dq_N}{q_N} + H(z,\tau) \, dz \mapsto \left(G(z,\tau)\big|_{q_N = 0}\right)\, \frac{dq_N}{q_N} + \left(H(z,\tau)\big|_{q_N = 0}\right) \, dz.
$$
Observe that
$$
A_{m,\alpha}(\tau) \big|_{q_N = 0} = -\frac{(m+1)}{(2\pi i)^{m+1}} G_{m + 2, \alpha}(\tau) \big|_{q_N = 0} = \frac{2\pi i (-1)^m B_{m+2}([x_\alpha])}{m!(m+2)}
$$
and
\begin{align*}
    g_\alpha(\bX,z|\tau) \big|_{q_N = 0} &= \frac{\partial}{\partial \bX} \left(\pi i e^{-\alpha\tau\bX}\left( \coth(\bX/2) + \coth(\pi i(z - \tilde{\alpha}))\right)- \frac{2\pi i}{\bX}\right) \cr
    &= \left(\pi i e^{-x_\alpha\bX} \frac{\partial}{\partial \bX} \coth(\bX/2)\right) \cr
    &\qquad\qquad \qquad - x_\alpha \pi i e^{-x_\alpha\bX}\left( \coth(\bX/2) + \coth(\pi i(z - \tilde{\alpha}))\right) + \frac{2\pi i }{\bX^2} \cr
    &= \sum_{m = 0}^\infty \frac{2\pi i (-1)^m B_{m+2}([x_\alpha])}{m!(m+2)}\bX^m.
\end{align*}
Let $w = e^{2\pi i z}$ and $\zeta = e^{2\pi i/N}$. Then 
\begin{align*}
    h_\alpha(\bX,z|\tau) &= \pi i e^{-\alpha\tau\bX}\left( \coth(\bX/2) + \coth(\pi i(z - \tilde{\alpha}))\right)- \frac{2\pi i}{\bX} \cr
    &= \pi ie^{-x_\alpha\bX} \left(\frac{e^\bX + 1}{e^\bX - 1} + \frac{wq_N^{-x_\alpha}\zeta^{-y_\alpha} + 1}{wq_N^{-x_\alpha}\zeta^{-y_\alpha} -1}\right) - \frac{2\pi i}{\bX} \cr
    &= 2\pi i e^{-x_\alpha\bX}\left(\frac{1}{e^\bX - 1} + \frac{wq_N^{-x_\alpha}\zeta^{-y_\alpha}}{wq_N^{-x_\alpha}\zeta^{-y_\alpha} -1}\right)- \frac{2\pi i}{\bX}.
\end{align*}
Thus,
$$
h_\alpha(\bX,z|\tau)\big|_{q_N = 0} = \begin{cases}
2\pi i \left(\frac{1}{e^\bX - 1} + \frac{w}{w -
\zeta^{y_\alpha}}\right)- \frac{2\pi i}{\bX} & x_\alpha = 0 \cr
2\pi i e^{-x_\alpha\bX}\left(\frac{1}{e^\bX - 1} + 1 \right)- \frac{2\pi i}{\bX} & x_\alpha \neq 0.
\end{cases}
$$
Summing these gives the restriction to the identity component of the singular fiber
\begin{multline*}
    \Omega_N' = N\left(\bY \frac{\partial }{\partial\bX} + \frac{1}{2}\sum_{\substack{\alpha \in (N^{-1}\Z/\Z)^2 \\ m \ge 0}} \frac{(-1)^m B_{m+2}([x_\alpha])}{m!(m+2)} \e_{m+2,\alpha}\right) \frac{dq_N}{q_N} \cr
    + \sum_\alpha \frac{e^{-x_\alpha\bX}}{e^\bX - 1} \cdot \bt_\alpha \, \frac{dw}{w} + \sum_{x_\alpha \in \Z} \bt_\alpha \, \frac{dw}{w - e^{2\pi iy_\alpha}} + \sum_{x_\alpha \notin \Z} e^{-x_\alpha \bX} \cdot \bt_\alpha \, \frac{dw}{w}.
\end{multline*}
\begin{theorem}
\label{thm:KZ}
The pullback of the KZB connection along the inclusion of $\P^1 - \{0,\bmu_N,\infty\}$ into the identity component of the singular fiber of $\overline{E}_N$ above $q_N = 0$ is the cyclotomic KZ connection \eqref{eqn:KZ} after the change of variables mapping
$$
\ee_0 \longmapsto \displaystyle \sum_\alpha \frac{e^{-x_\alpha\bX}}{e^\bX - 1} \cdot \bt_\alpha +  \sum_{x_\alpha \notin \Z} e^{-x_\alpha \bX} \cdot \bt_\alpha
$$
and if $\zeta = e^{2\pi i k/N}$, then $\ee_\zeta \longmapsto \bt_\alpha$ where $x_\alpha \equiv 0$ and $y_\alpha \equiv k/N$. This change of variables is a full level $N$ generalization of Hain's map \cite[\S 18]{hain:kzb}.
\end{theorem}

Meanwhile, the $dq_N$ component of the linearized connection $\Omega_N'$ is constant with respect to $w$. Thus, the residue of $\Omega_N$ along the fiber above $q_N = 0$ is the derivation
$$
L_{q_N=0} := \Res_{q_N = 0}\Omega_N = N\left(\bY \frac{\partial }{\partial\bX} + \frac{1}{2}\sum_{\substack{\alpha \in (N^{-1}\Z/\Z)^2 \\ m \ge 0}} \frac{(-1)^m B_{m+2}([x_\alpha])}{m!(m+2)} \e_{m+2,\alpha}\right).
$$
One can use modularity properties \eqref{eqn:hinv} and \eqref{eqn:ginv} to compute the restriction and residues along the identity components of the other singular fibers. If $P$ is a cusp such that $\gamma \in SL_2(\Z)$ maps $q_N = 0$ to $P$, then the residue along the fiber above $P$ is 
$$
L_{q_N,P} = N\left(\bY \frac{\partial }{\partial\bX} + \frac{1}{2}\sum_{\substack{\alpha \in (N^{-1}\Z/\Z)^2 \\ m \ge 0}} \frac{(-1)^m B_{m+2}([x_\alpha])}{m!(m+2)} \e_{m+2,\alpha\gamma^{-1}}\right).
$$

\section{Admissible variation of MHS}
\label{sec:avmhsP}

We now prove the main result. We assume familiarity with mixed Hodge structures and begin with a review of variations of MHS. The reader can find relevant background in \cite{SZ,kashiwara}.

\subsection{Review of variations of MHS}

Suppose $X$ is a smooth projective variety over $\C$ and $D$ is a divisor with normal crossings in $X$. Let $Y = X - D$. Relevant examples are: 
\begin{itemize}
    \item where $Y$ is a modular curve, $X$ is its natural compactification, and $D$ is the set of cusps;
    \item where $Y$ is the universal elliptic curve $\E_\G'$, $X$ is the compactification $\overline{\E}_\G$, and $D$ is the union of $\E_\G[N]$ and singular fibers of $\overline{\E}_\G$ over the cusps of $Y_\G$.
\end{itemize}

Let $\V$ be a $\Q$-local system of finite rank over $Y$ with unipotent local monodromy at every smooth point of $D$. Let $\cV = \V \otimes_\Q \cO_X$ be the associated flat vector bundle. Denote Deligne's canonical extension of $\cV$ to $X$ by $\overline{\cV}$. Then $\overline{\cV}$ has natural connection 
$$
\nabla : \overline{\cV} \to \overline{\cV} \otimes \Omega_X^1(\log D)
$$
with logarithmic singularities along $D$. Since the local monodromy operators are unipotent, the residues of $\nabla$ at each smooth point of $D$ are nilpotent.

\begin{definition}
A {\em variation of MHS} $\V$ over $Y$ consists of a local system $\V_\Q$ over $Y$ of finite dimensional rational vector spaces with unipotent local monodromy, together with
\begin{enumerate}[(i)]
    \item a finite increasing filtration $W_\bullet$ of $\V_\Q$ by $\Q$-local systems 
    $$
    0 \subseteq W_a \V \subseteq \cdots \subseteq W_{r - 1}\V \subseteq W_r\V \subseteq \cdots \subseteq W_b\V = \V,
    $$
    and
    \item a finite decreasing filtration $F^\bullet$ of $\cV$ by holomorphic subbundles.
\end{enumerate}
These are required to statisfy:
\begin{enumerate}[(i)]
    \item {\em Griffiths' transversality}:
    $$
    \nabla(F^p\cV) \subseteq F^{p-1}\cV\otimes \Omega^1_Y = F^p(\cV \otimes \Omega^1_Y).
    $$
    \item The fiber $V_y$ above any point $y \in Y$ is a MHS with weight and Hodge filtrations cut out by $W_\bullet$ and $F^\bullet$, respectively.
\end{enumerate}
\end{definition}
\begin{definition}
\label{def:admissible}
We continue with the notation above. Suppose for the time being that $X$ is a curve. A variation of MHS $\V$ over $Y$ is {\em admissible} if the following additional conditions hold. 
\begin{enumerate}[(i)]
    \item The subbundles $F^p\cV$ extend to holomorphic subbundles of the canonical extension $\overline{\cV}$. 
    \item For $P \in D$, let $L_P = -\Res_P \nabla$ and $V_P$ be the fiber of $\overline{\cV}$ above $P$. There exists an increasing {\em relative weight filtration} $M_\bullet$ of $V_P$ such that 
    \begin{enumerate}[(a)]
        \item $L_P(M_rV_P) \subseteq M_{r - 2}V_P$ and $L_P(W_mV_P) \subseteq W_mV_P$ for all $m$ and $r$, and
        \item $L_P^r$ induces an isomorphism 
        \begin{equation}
            \label{eqn:Lr}
            L_P^r : \Gr^M_{m + r} \Gr^W_mV_P \to \Gr_{m - r}^M \Gr^W_mV_P
        \end{equation}
        for all $m$ and $r$. 
    \end{enumerate}
\end{enumerate}
When $\dim X > 1$, a variation $\V$ over $Y$ is admissible if the pullback of $\V$ to the normalization of every curve $C$ in $X$ is admissible.
\end{definition}

If $X$ is a curve and $\V$ is admissible, the fibers $V_P$ over $P \in D$ have canonical {\em limit MHS} for each choice of a nonzero tangent vector $\vv \in T_PX$. We typically denote this MHS by $V_{P,\vv}$ (or $V_\vv$ if the choice of point is clear). The $\Q$-structure of $V_{P,\vv}$ is determined by the elements 
\begin{equation}
    \label{eqn:Qstructure}
    \lim_{t \to 0} t^{-L_P} v(t) \in V_P, 
\end{equation}
where $t$ is the local holomorphic coordinate of $X$ centered at $P$ such that $\vv = \partial/\partial t$ and $v(t)$ is a local flat section of $\V_\Q$. The weight and Hodge filtrations of $V_{P,\vv}$ are $M_\bullet$ and the restriction of $F^\bullet$ to $V_P$, respectively. 

\begin{example}
The pullback $\H_\G$ of the local system $\H$ to $Y_\G$ is admissible. The limit MHS at $\partial/\partial q$ at $q = 0$ is isomorphic to $\Q(0) \oplus \Q(1)$ with $\Q$-Betti basis $\bX$ and $\aa$.
\end{example}

\subsection{Admissibility of $\bP_N$}

We now prove our second main result. Together with appropriate choices of Hodge, weight, and relative weight filtrations, the KZB connection underlies an admissible variation of MHS over $\E_N'$. 

The natural Hodge and weight filtrations on the Lie algebra $\p_N$ are as follows. 
\begin{align*}
    F^{-p} &= \{x \in \p_N \mid \deg_\bY(x) + \sum\deg_{\bt_\zeta}(x) \leq p\} \cr
    W_{-m} &= \{x \in \p_N \mid \deg_\bY(x) + \deg_\bX(x) + 2 \sum\deg_{\bt_\zeta}(x) \geq m\}
\end{align*}
There is also a relative weight filtration $M_\bullet$ along the singular fibers of $\E_N'$
$$
M_{-m} = \{x \in \p_N \mid 2\deg_\bY(x) + 2\sum\deg_{\bt_\zeta}(x) \geq m\}
$$
These filtrations are compatible with the filtrations on $\H$ and the Lie bracket. The $\Q$-structure of each fiber $\p_N$ of $\bP_N$ is determined by the image of the $\p(E_\tau',x)$ under the isomorphism $\psi_\tau : \p(E_\tau',x) \otimes_\Q \C \to \p_N$ in Proposition \ref{prop:fiber}.
\begin{lemma}
The Hodge, weight, and relative weight filtrations of $\p_N$ are well-defined filtrations of the vector bundle $\mathbfcal{P}_N \to \E_N'$.
\end{lemma}
\begin{proof}
The factors of automorphy $\widetilde{M}_{\gamma,(m,n)}$ are in $F_0W_0M_0\End \p_N$.
\end{proof}

\begin{theorem}
\label{thm:avmhs}
Together with these filtrations, $\mathbfcal{P}_N^\topo \to \E_N'$ is a pro-object of the category of admissible variations of MHS. The $W_\bullet$-graded quotients of $\mathbfcal{P}_N^\topo$ are direct sums of Tate twists of symmetric powers of $\H_{\G(N)}$, and the relative weight filtration of the limit MHS at $\partial/\partial q_N + \partial/\partial w$ anchored at the point $(q_N = 0, w = 1)$ is $M_\bullet$.
\end{theorem}
\begin{proof}
We observed in the previous section that $\nabla_{\KZB_N}$ has residues 
$$
L_{z,\tilde{\alpha}} = e^{-x_\alpha\bX}\cdot \bt_\alpha
$$
and 
$$
L_{q_N,P} = N\left(\bY \frac{\partial }{\partial\bX} + \frac12 \sum_{\substack{\alpha \in (N^{-1}\Z/\Z)^2 \\ m \ge 0}} \frac{(-1)^m B_{m+2}([x_\alpha])}{m!(m+2)} \e_{m+2,\alpha\gamma^{-1}}\right),
$$
where $\gamma \in SL_2(\Z)$ such that $\gamma P$ is the cusp under $q_N = 0$. Thus, we have 
$$
L_{z,\tilde{\alpha}} \in F^{-1}M_{-2}W_{-2} \Der \p_N.
$$
and
$$
L_{q,P} \in F^{-1}M_{-2}W_0 \Der \p_N.
$$
Therefore, $\nabla_{\KZB_N}$ satisfies Griffiths transversality. Meanwhile, $L_{q,P}$ and $L_{q,P} + L_{w,\tilde{\alpha}}$ satisfies part (ii)(a) of Definition \ref{def:admissible}. 

The graded quotient $\Gr_{-m}^W \bP_N$ is a direct sum of variations of the form 
$$
S^{m-2k}\cH_{\G(N)} \cdot f_k,
$$
where $\cH_{\G(N)} := \H_{\G(N)} \otimes \cO_{Y(N)}$ and $f_k$ is a Lie word in the terms $\bt_\alpha$ of length $k$. Each $\bt_\alpha$ cuts out a constant section of $\bP_N$ (trivial under the monodromy action of $\G(N) \ltimes \Z^2$). Thus each Lie word in the $\bt_\alpha$ of length $k$ spans a constant local system with fiber $\Q(k)$. Thus, $\Gr_{-m}^W \bP_N^\topo$ is a direct sum of variations of the form $S^{m-2k}\H(k)$. We then observe $\Gr^W_0 L_{q,P} = N\bY\partial/\partial\bX$. It follows from the representation theory of $\sl_2$ that $L_{q,P}^r$ induces an isomorphism
$$
L_{q,P}^r : \Gr_{-m+r}^M\Gr_{-m}^W \p_N \to \Gr_{-m-r}^M\Gr^W_{-m} \p_N.
$$
Since $L_{w,\tilde{\alpha}} \in W_{-2}\Der \p_N$, the sum $(L_{q,P} + L_{w,\tilde{\alpha}})^r$ induces the same isomorphism. Thus, relative weight filtration $M_\bullet$ satisfies \eqref{eqn:Lr} at every point along the singular fibers of $\E_N$.

At points along the $N$-torsion sections of $\E_N$, the residues $L_{z,\tilde{\alpha}}$ are contained in $W_{-2}\Der \p_N$. Thus, the relative weight filtration is simply equal to $W_\bullet$ and \eqref{eqn:Lr} is satisfied trivially. 
\end{proof}

\begin{cor}
The MHS on the fiber of $\bP_N$ over $(E',x)$ is the canonical MHS on $\p(E',x)$. 
\end{cor}
\begin{proof}
The monodromy action $\pi_1(E',x) \to \Aut \piun(E',x)$ of $\bP_N^\topo$ is induced by conjugation. The induced representation on Lie algebras $\theta : \p(E',x) \to \Der \p(E',x)$ is given by the adjoint action. Denote by $\p(E',x)^\KZB$ the MHS on $\p(E',x)$ with the $\Q$-structure induced by Proposition \ref{prop:fiber} and Hodge and weight filtrations given at the beginning of this section. Denote $\p(E',x)$ with its canonical MHS by $\p(E',x)^\can$.

Since the restriction of $\bP_N$ to the fiber $E'$ is unipotent and admissible, it follows from the main result of \cite{HZ} that
$$
\theta : \p(E',x)^\can \to \Der \p(E',x)^\KZB
$$
is a morphism of MHS. Meanwhile, 
$$
\theta : \p(E',x)^\KZB \to \Der \p(E',x)^\KZB
$$
is also a morphism of MHS since $\p(E',x)^\KZB$ is an object in the category of pro-MHS. Finally, since $\p(E',x)$ is free and thus has trivial center, the representation $\theta$ is injective. It follows from the strictness of morphisms of MHS that $\p(E',x)^\KZB$ and $\p(E',x)^\can$ must be the same. 
\end{proof}

Since $\bP_N$ is admissible, the result extends to the boundary of $\E_N'$. If $\ww = \lambda \partial/\partial q_N + \mu\partial/\partial w$ is a tangent vector anchored at $(q_N = 0,w=1)$, the limit MHS of $\bP_N$ at $\ww$ is the canonical MHS on $\Lie\piun(E_{\lambda\partial/\partial q_N},\mu\partial/\partial w)$, where $E_{\lambda\partial/\partial q_N}'$ is the first order smoothing of the singular fiber above $q_N = 0$ in the direction of $\lambda \partial/\partial q_N$ with its $N$-torsion removed and $\mu \partial/\partial w$ is a tangential base point anchored at the identity of $E_{\lambda\partial/\partial q_N}$.

\section{$\G_1(N)$ and other congruence subgroups}
\label{sec:other}

The above formulas and proofs can easily be adapted to any level $N$ congruence subgroup $\G$. One simply replaces $\E_N'$ by $\E_\G'$ and $\p_N$ by the Lie algebra
$$
\p_\G := \bL\left(\bX,\bY,\bt_\alpha \mid \alpha \in \E_\G[N]\right)^\wedge \Big/ \Big(\sum_\alpha \bt_\alpha = [\bX,\bY]\Big).
$$
This is equivalent to setting $\bt_\alpha$ to zero in the connection form for all $\alpha \notin \E_\G[N]$. The resulting connection is meromorphic on $\E_\G$ with at worst logarithmic singularities along $\E_\G[N]$ and the singular fibers.

\begin{example}
In the case of the full modular group $\G(1) = \SL_2(\Z)$, the only nontrivial $\bt_\alpha$ is $\bt_0$. Then the Lie algebra $\p_1$ is free on the two generators $\bX$ and $\bY$. The connection form $\Omega_1$ reduces to Hain's formula \cite[\S 9.2]{hain:kzb}. 
\end{example}

\begin{remark}
To recover the bilevel $(M,N)$ connection of \cite{CG}, one only needs to include the sections of $\E_{\lcm(M,N)}$ that are invariant under the action of the subgroup 
$$
\G(M,N) := \left\{\begin{pmatrix} a & b \cr c & d \end{pmatrix} \middle| a \equiv 0 \bmod M, d \equiv 0 \bmod N\right\} \subset \SL_2(\Z).
$$
\end{remark}

\subsection{KZB for $\G_1(N)$} 

The congruence subgroup $\G_1(N)$ is of particular interest in the study of mixed Tate motives over $\Z[\bmu_N,1/N]$. In this case, the KZB connection only includes $\bt_\alpha$ where $x_\alpha = 0$. These terms correspond via Proposition \ref{prop:fiber} to small loops around the real $N$-torsion of an elliptic curve $\C/\Lambda_\tau$. Since these torsion sections intersect the singular fiber above $\tau = i\infty$ at $N$th roots of unity, we will reindex the connection by $\zeta \in \bmu_N$ instead of $\alpha$. If $Ny_\alpha \equiv \ell \mod N$, then $\bt_\alpha = \bt_\zeta$ where $\zeta = e^{2\pi i\ell/N}$. The reason for this choice in notation becomes clear after pulling the connection back to $\P^1 - \{0,\bmu_N,\infty\}$ in \S\ref{sec:hain}. 

The $\G_1(N)$ KZB connection is then $\nabla_{\G_1(N)} = d + \Omega_{\G_1(N)}$, where 
\begin{multline*}
\Omega_{\G_1(N)} = \left(2\pi i\bY\frac{
\partial}{\partial \bX} + \frac{1}{2}\sum_{\substack{m \geq 0 \\ \zeta \in \bmu_N}}A_{m,\zeta}(\tau)\delta_{m,\zeta} + \sum_{\zeta \in \bmu_N} g_\zeta(\bX,z|\tau) \cdot \bt_\zeta\right)d\tau
\cr
+ \left(2\pi i\bY + \sum_{\zeta \in \bmu_N} h_\zeta(\bX,z|\tau) \cdot \bt_\zeta\right) \, dz.
\end{multline*}
The restriction of the $\G_1(N)$ connection to the zero section is
$$
\Omega_{\G_1(N)}' = 2\pi i\left(\bY\frac{\partial}{\partial \bX} - \frac{1}{2} \sum_{\substack{m\ge 2 \\ \zeta \in \bmu_N}} \frac{\zetabar + (-1)^m \zeta}{(m-2)!} \bG_{m,\zeta}(\tau)\e_{m,\zeta}\right)\, d\tau + \bt_1 \frac{dz}{z}.
$$
Letting $w = e^{2\pi i z}$, the resides along the torsion section $w = \zeta$ for $\zeta \in \bmu_N$ are simply 
$$
L_{w, \zeta} := \Res_{w = \zeta}\Omega_{\G_1(N)} = \bt_\zeta.
$$
Similarly, the restriction of the $\G_1(N)$ connection to a neighborhood of the singular fiber above $q_N = 0$ simplifies to
\begin{multline}
\label{eqn:KZBrest}
    \Omega_{\G_1(N)}' = \left(\bY \frac{\partial }{\partial\bX} + \frac12 \sum_{\substack{\zeta \in \bmu_N \\ m \ge 0}} \frac{B_{2m+2}}{(2m)!(2m+2)} \e_{2m+2,\zeta}\right) \frac{dq}{q} \cr
    + \frac{\bX}{e^\bX - 1} \cdot \bY \, \frac{dw}{w} + \sum_{\zeta \in \bmu_N} \bt_\zeta \, \frac{dw}{w - \zeta} .
\end{multline}
We replace $q_N$ with $q$ since the cusp under $q = q_N = 0$ has width 1 in $Y_1(N)$. The resulting residue along the nodal cubic singular fiber above $q = 0$ is 
\begin{equation}
    \label{eqn:res}
    L_{q=0} := \Res_{q = 0}\Omega_{\G_1(N)} = \bY \frac{\partial }{\partial\bX} + \sum_{\substack{\zeta \in \bmu_N \\ m \ge 0}} \frac{B_{2m+2}}{(2m)!(2m+2)} \e_{2m+2,\zeta}.
\end{equation}
\begin{remark}
Since the KZB connection is flat and the sections $w = \zeta$ intersect the singular fiber above $q = 0$ transversely, we expect the residues $L_{w, \zeta}$ and $L_{q, 0}$ to commute. This can be confirmed by a straightforward calculation of $[L_{q,0},L_{w,\zeta}] = L_{q,0}(\bt_\zeta) = 0$. 
\end{remark}

\subsection{Pullback to $\P^1 - \{0,\bmu_N,\infty\}$}
\label{sec:hain}

By Theorem \ref{thm:KZ}, the $\G_1(N)$ KZB connection degenerates to the cyclotomic KZ connection at the singular fiber above $q = 0$, but the change of variables is simpler than the full level $N$ case.

Pull back the KZB connection to $\P^1$ along the inclusion $U_N := \P^1 - \{0,\bmu_N,\infty\} \to E_{\partial/\partial q}$, the fiber of $\E_{\G_1(N)}$ above the tangent vector $\partial/\partial q$ at $q = 0$. This is equivalent to setting $\frac{dq}{q}$ to 0 in \eqref{eqn:KZBrest}, which yields
$$
\Omega_{U_N} = \frac{\bX}{e^\bX - 1} \cdot \bY \, \frac{dw}{w} + \sum_{\zeta \in \bmu_N} \bt_\zeta \, \frac{dw}{w - \zeta}
$$
Upon the change of variables
\begin{equation*}
    \left\{ 
    \begin{array}{lll}
        \ee_0 & \longmapsto & \frac{\bX}{e^{\bX} - 1} \cdot \bY \\
        \ee_\zeta & \longmapsto & \bt_\zeta,
    \end{array}
    \right.
\end{equation*}
this is exactly the cyclotomic KZ connection \eqref{eqn:KZ}. 

\appendix

\part*{Appendices}

\section{Invariance of the KZB connection}
\label{sec:invariance}

Here we prove Proposition \ref{prop:invariance}, the invariance of the level $N$ KZB connection under the action of $\SL_2(\Z) \ltimes \Z^2$. The calculation is similar to the proof of the level 1 case in \cite{LR}. 

\subsection{Review of Lie theory} 

First recall some elementary Lie theory. Let $V$ be a vector space and $\varphi,\phi \in \End V$. Then 
\begin{equation}
    \label{eqn:expdot}
    \exp(\ad \phi) (\varphi) = e^\phi \circ \varphi \circ e^{-\phi}. 
\end{equation}
Thus, in our notation, if $\delta \in \Der \p_N$ and $\phi \in \C\ll \bX,\bY,\bt_\alpha\rr$, then 
$$
\exp(\phi) \cdot \delta = e^\phi \circ \delta \circ e^{-\phi}.
$$
\begin{lemma}
If $f \in \C\ll \bX,\bY,\bt_\alpha\rr$ and $\delta$ is a continuous derivation of $\C\ll \bX,\bY,\bt_\alpha\rr$, then 
\begin{equation}
    \label{eqn:contder}
    e^{-f}\delta(e^f) = \frac{1 - \exp(-\ad_f)}{\ad_f}\delta(f).
\end{equation}
\end{lemma}
\begin{lemma}    
\label{lem:lieexp}
If $\delta \in \Der \p_N$ and $m \in \C$, then 
\begin{equation}
    e^{-m\bX} \cdot \delta = \delta + \frac{1-e^{-m\bX}}{\bX}\delta(\bX).
\end{equation}
\end{lemma}
\begin{cor}
\label{cor:ad}
For all $\v \in \p_N$ and $m \in \C$, 
$$
e^{-m\bX} \cdot (\ad \v) = \ad(e^{-m\bX} \cdot \v).
$$
\end{cor}

\subsection{Proof of invariance}

The following statements prove Proposition \ref{prop:invariance}. We proceed by first showing elliptic invariance with respect to $\Z^2$ and then modular invariance with respect to $\SL_2(\Z)$. The following properties of the Jacobi form $F$ are used extensively.
\begin{itemize}
    \item ({\em Symmetry property}) 
    $$
    F(x,z,\tau) = F(2\pi iz,x/(2\pi i),\tau) = -F(-x,-z,\tau)
    $$
    \item ({\em Elliptic property})
    $$
    F(x,z + m\tau + n,\tau) = e^{-mx}F(x,z,\tau)
    $$    
    for all $m,n\in\Z$.
    \item ({\em Modularity property})
    $$
    F\left(\frac{x}{c\tau + d},\frac{z}{c\tau + d},\gamma\tau\right) = (c\tau+d)\exp\left(\frac{czx}{c\tau + d}\right)F(x,z,\tau)
    $$
    for all $\gamma \in \SL_2(\Z)$.
\end{itemize}
Proof of these properties can be found in \cite[\S3]{zagier}.

\begin{lemma} For all $(m,n) \in \Z^2$, 
$$
(m,n)^\ast \Omega_N = e^{-m\bX}\cdot \Omega_N.
$$
\end{lemma}
\begin{proof}
The forms $\bY\partial/\partial\bX \, d\tau$ and $\psi$ are invariant with respect to $\Z^2$. It follows from Lemma \ref{lem:lieexp} that 
\begin{align}
\label{eqn:Hmn}
(m,n)^\ast \left(2\pi i\bY \frac{\partial}{\partial \bX} \, d\tau \right) &= 2\pi i\bY \frac{\partial}{\partial \bX} \, d\tau  \cr
&=  e^{-m\bX}\cdot \left(2\pi i\bY \frac{\partial}{\partial \bX} \right) \, d\tau - 2\pi i \frac{1 - e^{-m\bX}}{\bX}\bY
\end{align}
and
\begin{equation}
    \label{eqn:psimn}
    (m,n)^\ast \psi = \psi = e^{-m\bX} \psi.
\end{equation}
Meanwhile,
\begin{align*}
    (m,n)^\ast \left(g_\alpha(\bx, \right. & \left. z|\tau)\cdot \bt_\zeta \, d\tau \right) \cr
    &= g_\alpha(\bx,z+m\tau+n|\tau)\cdot \bt_\alpha \, d\tau \cr
    &= \frac{\partial}{\partial \bX} \left(e^{-x_\alpha\bX}F(\bX, z + m\tau +n - \tilde{\alpha},\tau) - \frac{2\pi i}{\bX} \right) \cdot \bt_\alpha \, d\tau \cr
    &= \frac{\partial}{\partial \bX} \left(e^{-(x_\alpha+m) \bX} F(\bX, z -\tilde{\alpha},\tau) - \frac{2\pi i}{\bX} \right) \cdot \bt_\alpha \, d\tau \cr
    &= \left(-(x_\alpha+m) e^{-(x_\alpha+m)\bX} F(\bX, z - \tilde{\alpha},\tau) \right. \cr 
    &\qquad \qquad \qquad \qquad \left. + e^{-(x_\alpha+m) \bX} \frac{\partial F}{\partial \bX}(\bX,z -\tilde{\alpha},\tau)  + \frac{2\pi i}{\bX^2}\right)\cdot \bt_\alpha \, d\tau \cr
    &= \left(e^{-m\bX}\frac{\partial g_\alpha}{\partial\bX}(\bX,z|\tau) - me^{-(x_\alpha+m)\bX}F(\bX,z-\tilde{\alpha},\tau) + \frac{2\pi i}{\bX^2}\right)\cdot \bt_\alpha \, d\tau
\end{align*}
and
\begin{align*}
    (m,n)^\ast \left(h_\alpha(\bX,z|\tau) \right. & \left. \cdot \bt_\alpha\, dz\right) \cr 
    &= \left(e^{-x_\alpha \bX} F(\bX, z + m\tau + n - \tilde{\alpha}, \tau) - \frac{2\pi i}{\bX}\right)\cdot \bt_\alpha \, d(z + m\tau + n) \cr 
    &= \left(e^{-(x_\alpha+ m)\bX} F(\bX, z -\tilde{\alpha},\tau)- \frac{2\pi i}{\bX}\right) \cdot \bt_\alpha \, (dz + m\,d\tau)
\end{align*}
Then
\begin{align}
\label{eqn:nunm}
    (m,n)^\ast (\nu_1 + \nu_2) &= 2\pi i\bY + \sum_{\alpha \in (N^{-1}\Z/\Z)^2} \left(e^{-m\bX}\frac{\partial g_\alpha}{\partial\bX}(\bX,z|\tau) + \frac{2\pi i}{\bX^2}\right)\cdot \bt_\alpha \, d\tau \cr 
    &\qquad \qquad \qquad \qquad + \sum_{\alpha \in (N^{-1}\Z/\Z)^2} \left(e^{-m\bX}h_\alpha(\bX,z|\tau) - \frac{2\pi i}{\bX}\right) \cdot \bt_\alpha \, dz \cr
    &= e^{-m\bX}\cdot \nu_1 + 2\pi i\frac{1 - e^{-m\bX}}{\bX}\bY + e^{-m\bX}\cdot \nu_2.
\end{align}
Summing \eqref{eqn:Hmn}, \eqref{eqn:psimn}, and \eqref{eqn:nunm} proves the lemma.
\end{proof}
\begin{cor}
For $(m,n) \in \Z^2$ and $I \in SL_2(\Z)$ the identity matrix,
$$
(m,n)^\ast\Omega_N = \mathrm{Ad}(\widetilde{M}_{I,(m,n)})\Omega_N - d\widetilde{M}_{I,(m,n)}\widetilde{M}_{I,(m,n)}^{-1}.
$$
\end{cor}
\begin{proof}
For all $(m,n)\in\Z^2$, the factor of automorphy is $\widetilde{M}_{I,(m,n)} = \exp(-m\bX)$. Thus, 
\begin{align*}
\mathrm{Ad}(\widetilde{M}_{I,(m,n)})\Omega_N &= \exp(-m\bX) \circ \Omega_N \circ \exp(m\bX) \cr
&= \exp(-m\bX) \cdot \Omega_N \cr
&= (m,n)^\ast \Omega_N. 
\end{align*}
The factor of automorphy does not depend on $z$ or $\tau$ and thus $d\widetilde{M}_{I,(m,n)} = 0$.
\end{proof}
We next show $\nabla_{\KZB_N}$ is invariant with respect to $\SL_2(\Z)$. We shall abbreviate $\widetilde{M}_{\gamma,(0,0)}$ to simply $\widetilde{M}_\gamma$. 
\begin{lemma}
\label{lem:gpsi}
For all $\gamma \in \SL_2(\Z)$, 
\begin{equation*}
\label{eqn:gpsi}
\gamma^\ast \psi = \Ad(\widetilde{M}_\gamma)\psi.
\end{equation*}
\end{lemma}
\begin{proof}
By \eqref{eqn:Amodular}, we have 
\begin{align*}
\gamma^\ast\left(A_{m,\alpha}(\tau)\delta_{m,\alpha} \, d\tau\right) &= A_{m,\alpha}(\gamma\tau)\delta_{m,\alpha} \, d(\gamma\tau) \cr 
&= (c\tau + d)^{m+2}A_{m,\alpha\gamma}(\tau)\delta_{m,\alpha} \, \frac{d\tau}{(c\tau+d)^2} \cr
&= (c\tau + d)^m A_{m,\alpha\gamma}(\tau)\delta_{m,\alpha} \, d\tau.
\end{align*}
Meanwhile, we know $e^{cz\bX/(c\tau+d)} \cdot \delta_{m,\alpha}$ since $\delta_{m,\alpha}$ annihilates $\bX$. Thus, we have 
\begin{align*}
    \Ad(\widetilde{M}_\gamma(\tau))\delta_{m,\alpha} &= \Ad(M_\gamma(\tau)) \delta_{m,\alpha} \cr
    &= (c\tau + d)^m \delta_{m,\alpha\gamma^{-1}}.
\end{align*}
Summing over $\alpha \in (N^{-1}\Z/\Z)^2$ yields the result.
\end{proof}
\begin{lemma}
For $T = \mat{1 & 1 \cr 0 & 1}$ and $S = \mat{0 & -1 \cr 1 & 0}$,
\begin{align*}
    T^\ast \nu_1 &= \sum_{\alpha} g_{\alpha T}(\bX,z|\tau) \cdot \bt_\alpha \, d\tau \cr 
    S^\ast \nu_1 &= \sum_{\alpha} \left(z e^{z\bX} h_{\alpha S}(\tau\bX,z|\tau) \cdot \bt_\alpha \right. + \left. \tau e^{z\bX} g_{\alpha S}(\tau \bX, z | \tau) \cdot \bt_\alpha \right) \, \frac{d\tau}{\tau} \cr 
    & \qquad + 2\pi i \frac{z\bX e^{z\bX} - e^{z\bX} + 1}{\bX} \cdot \bY \, \frac{d\tau}{\tau^2}
\end{align*}
 
\end{lemma}
\begin{proof}
Follows directly from \eqref{eqn:ginv}.
\end{proof}
\begin{lemma}
For all $\gamma \in \SL_2(\Z)$,
$$
\Ad(\widetilde{M}_\gamma(\tau))(\nu_1) = \sum_{\alpha} e^{cz\bX}g_{\alpha}((c\tau+d)\bX,z|\tau) \cdot \bt_{\alpha\gamma^{-1}} \, d\tau.
$$
\end{lemma}
\begin{proof}
Follows from the formula for the factor of automorphy and Corollary \ref{cor:ad}.
\end{proof}
\begin{lemma}
For $T = \mat{1 & 1 \cr 0 & 1}$ and $S = \mat{0 & -1 \cr 1 & 0}$,
\begin{align*}
    T^\ast \nu_2 &= 2\pi i \bY + \sum_\alpha h_{\alpha T} (\bX, z|\tau) \cdot \bt_\alpha \, dz \cr 
    S^\ast \nu_2 &= 2\pi i e^{z\bX} \cdot \bY + \tau \sum_\alpha h_{\alpha S} (\tau \bX | \tau) \cdot \bt_\alpha \left(\frac{dz}{\tau} - \frac{z \, d\tau}{\tau^2}\right).
\end{align*}
\end{lemma}
\begin{proof}
Follows directly from \eqref{eqn:hinv}.
\end{proof}
\begin{lemma}
For all $\gamma \in \SL_2(\Z)$,
$$
\Ad(\widetilde{M}_\gamma(\tau))(\nu_2) = e^{cz\bX} \cdot \left(\frac{2\pi i\bY}{c\tau + d} + c \bX\right) \, dz + \sum_\alpha e^{cz\bX} h_\alpha((c\tau+d)\bX,z|\tau)\cdot \bt_{\alpha\gamma^{-1}} \, dz.
$$
\end{lemma}
\begin{proof}
Follows from the formula for the factor of automorphy and Lemma \ref{cor:ad}.
\end{proof}
\begin{cor}
\label{cor:adnu}
For $T = \mat{1 & 1 \cr 0 & 1}$ and $S = \mat{0 & -1 \cr 1 & 0}$,
\begin{equation*}
\label{eqn:adnu}
    \begin{aligned}
        (T^\ast - \Ad(\widetilde{M}_T(\tau)))(\nu_1 + \nu_2) &= 0 \cr 
        (S^\ast - \Ad(\widetilde{M}_S(\tau)))(\nu_1 + \nu_2) &= -\bX \, dz - 2\pi i \frac{e^{z\bX} - 1}{\bX} \frac{d\tau}{\tau^2}.
    \end{aligned}
\end{equation*}
\begin{proof}
Sum the results of the four previous lemmas.
\end{proof}
\end{cor}
\begin{lemma}
\label{lem:adHinv}
For all $\gamma \in \SL_2(\Z)$,
\begin{equation*}
\label{eqn:adHinv}
\begin{aligned}
    \Ad(\widetilde{M}_\gamma(\tau))&\left(2\pi i\bY \frac{\partial}{\partial \bX} \, d\tau \right) \cr 
    &= \Ad(M_\gamma(\tau))\left(2\pi i\bY \frac{\partial}{\partial \bX} \, d\tau + 2\pi i\frac{1 - e^{cz\bX/(c\tau + d)}}{\bX}\cdot \bY \, d\tau \right) \cr
    &= \Ad(M_\gamma(\tau))\left(2\pi i\bY \frac{\partial}{\partial \bX} \, d\tau\right) \cr 
    &\qquad \qquad \qquad + 2\pi i\frac{1 - e^{cz\bX}}{\bX} \cdot \left(\bY+ \frac{ c(c\tau + d)}{2\pi i}\bX\right) \, \frac{d\tau}{(c\tau + d)^2} \cr
    &= \Ad(M_\gamma(\tau))\left(\frac{1}{2\pi i}\bY \frac{\partial}{\partial \bX} \, d\tau\right) \cr 
    &\qquad \qquad \qquad + \left(2\pi i\frac{1 - e^{cz\bX}}{\bX} \cdot \bY - c^2z(c\tau + d)\bX\right) \, \frac{d\tau}{(c\tau + d)^2}
\end{aligned}
\end{equation*}
\end{lemma}
\begin{lemma}
\label{lem:dM}
For all $\gamma \in \SL_2(\Z)$,
\begin{equation*}
\label{eqn:dM}
\begin{aligned}
    d\widetilde{M}_\gamma(\tau)\widetilde{M}_\gamma(\tau)^{-1} &= d(M_\gamma(\tau)e^{cz\bX/(c\tau + d)})e^{- cz\bX/(c\tau + d)}M_\gamma(\tau)^{-1} \cr
    &= \left(dM_\gamma(\tau)e^{cz\bX/(c\tau + d)} + c\bX M_\gamma(\tau)e^{ cz\bX/(c\tau + d)}\, dz\right. & \cr
    & \qquad \qquad \left. - \frac{c^2z  M_\gamma(\tau)}{(c\tau+d)^2}\bX e^{cz\bX/(c\tau + d)}\, d\tau \right) e^{-cz\bX/(c\tau + d)}M_\gamma(\tau)^{-1} \cr
    &= dM_\gamma(\tau)M_\gamma(\tau)^{-1} + c\bX \, dz - c^2z \bX \, \frac{d\tau}{(c\tau+d)}
\end{aligned}
\end{equation*}
\end{lemma}
\begin{prop}
The $\KZB_N$ connection is invariant with respect to $\SL_2(\Z)$. 
\end{prop} 
\begin{proof}
Sum the results of Corollary \ref{cor:adnu} and Lemmas \ref{lem:gpsi}, \ref{lem:adHinv}, and \ref{lem:dM} to observe for $\gamma = S$ and $T$,
\begin{multline*}
    \gamma^\ast\Omega_N - \Ad(\widetilde{M}_\gamma(\tau))\Omega_N + d\widetilde{M}_\gamma(\tau)\widetilde{M}_\gamma(\tau)^{-1} \cr = \gamma^\ast \left(2\pi i\bY \frac{\partial}{\partial \bX} \, d\tau \right)  - \Ad(M_\gamma(\tau))\left(2\pi i\bY \frac{\partial}{\partial \bX} \, d\tau \right) + dM_\gamma(\tau)M_\gamma(\tau)^{-1}.
\end{multline*}
The right hand side is zero by a simple calculation \cite[Lemma 9.16]{hain:kzb}.
\end{proof}
\begin{remark}
The above proof also holds for any congruence subgroup $\G$ of level $N$ if the included $N$-torsion $\alpha$ are closed under the action of $\G$. 
\end{remark}

\section{Index of notation}

\begin{longtable}{ll @{\extracolsep{\fill}} r}
$Y(N)$ & the moduli space of elliptic curves with level $N$ structure & p.~\pageref{not:YN}\cr
$\E_N$ & the universal elliptic curve over $Y(N)$ & p.~\pageref{not:EN} \cr
$\E_N[N]$ & the set of $N$-torsion sections of $\E_N$ & p.~\pageref{not:ENN} \cr
$E'$ & an elliptic curve $E$ minus its $N$-torsion points & p.~\pageref{not:Eprime} \cr
$\p_N$ & fiber of the KZB connection & p.~\pageref{not:pN} \cr
$\bX, \bY, \bt_\alpha$ & the generators of $\p_N$ & p.~\pageref{not:XY} \cr
$\h$ & the upper half plane & p.~\pageref{not:h} \cr
$\Gm$ & the multiplicative group of $\C$ & p.~\pageref{not:Gm} \cr
$\bmu_N$ & the set of $N$th roots of unity & p.~\pageref{not:bmu} \cr
$\H$ & the local system $R_1f_\ast \Q$ associated to $f : \E_\G \to Y_\G$ & p.~\pageref{not:H} \cr
$\cH$ & the holomorphic vector bundle $\H \otimes \cO_{\M_{1,1}}$ & p.~\pageref{not:cH} \cr
$\aa,\bb$ & the Betti $\Q$-basis of $\H$ & p.~\pageref{not:ab} \cr
$Y_\G$ & the modular curve $\G \bbs \h$ & p.~\pageref{not:YG} \cr
$\E_\G$ & the universal elliptic curve over $Y_\G$ & p.~\pageref{not:EG} \cr
$X_\G$ & the compactification of $Y_\G$ & p.~\pageref{not:XG} \cr
$\overline{\E}_\G$ & the compactification of $\E_\G$ & p.~\pageref{not:EGbar} \cr
$\E_\G[N]$ & the set of $N$-torsion sections of $\E_\G$ & p.~\pageref{not:EGN} \cr
$\E_\G'$ & $\E_\G$ minus single-valued $N$-torsion & p.~\pageref{not:EGminus} \cr
$E_0$ & the nodal cubic & p.~\pageref{not:nodal} \cr
$G_{m,\alpha}$ & $\G(N)$ Eisenstein series of weight $m$ & p.~\pageref{not:Gma} \cr
$G_{m,\zeta}$ & $\G_1(N)$ Eisenstein series of weight $m$ & p.~\pageref{not:Gmz} \cr 
$\bL(S)$ & the free Lie algebra on the set $S$ & p.~\pageref{not:free} \cr 
$\piun(X,x)$ & the unipotent completion of $\pi_1(X,x)$ & p.~\pageref{not:piun} \cr
$\p(X,x)$ & the Lie algebra of $\piun(X,x)$ & p.~\pageref{not:p} \cr
$\e_{m,\alpha}$ & a derivation of $\p_N$ indexed by an $N$-torsion section $\alpha$ & p.~\pageref{eqn:ep} \cr
\end{longtable}

\end{document}